\documentclass{amsart}
\usepackage{xspace}
\numberwithin{figure}{section}
\numberwithin{equation}{section}
\usepackage[all]{xy}
\usepackage{tikz}
\usepackage{color}
\usepackage{amssymb}
\usepackage{comment}
\usepackage{enumitem,hyperref}
\usepackage{faktor}
\let\cal\mathcal
\def\Ascr{{\cal A}}

\def\Dscr{{\cal D}}
\def\Escr{{\cal E}}
\def\Fscr{{\cal F}}

\def\Lscr{{\cal L}}
\def\Mscr{{\cal M}}

\def\Oscr{{\cal O}}
\def\Pscr{{\cal P}}

\def\Sscr{{\cal S}}

\def\Wscr{{\cal W}}
\def\Xscr{{\cal X}}

\let\blb\mathbb

\def\QQ{{\blb Q}}

\def \ZZ{{\blb Z}}

\def \NN{{\blb N}}
\def \RR{{\blb R}}

\def\Res{\operatorname{Res}}

\def\Lotimes{\overset{L}{\otimes}}
\def\quot{/\!\!/}

\def\Qch{\operatorname{Qch}}
\def\coh{\mathop{\text{\upshape{coh}}}}

\def\Spec{\operatorname {Spec}}

\def\GL{\operatorname {GL}}

\def\diag{\operatorname {diag}}

\def\Hom{\operatorname {Hom}}
\def\uHom{\operatorname {\mathcal{H}\mathit{om}}}
\def\uEnd{\operatorname {\mathcal{E}\mathit{nd}}}
\def\End{\operatorname {End}}
\def\REnd{\operatorname {REnd}}
\def\RHom{\operatorname {RHom}}
\def\uRHom{\operatorname {R\mathcal{H}\mathit{om}}}
\def\uREnd{\operatorname {R\mathcal{E}\mathit{nd}}}

\def\relint{\operatorname {relint}}
\def\im{\operatorname {im}}

\def\ker{\operatorname {ker}}

\def\End{\operatorname {End}}

\def\Pic{\operatorname {Pic}}

\def\gldim{\operatorname {gl\,dim}}

\def\r{\rightarrow}

\DeclareMathOperator{\Proj}{Proj}

\DeclareMathOperator{\Ind}{Ind}

\DeclareMathOperator{\Aut}{Aut}

\DeclareMathOperator{\Perf}{Perf}
\def\open{H}
\DeclareMathOperator{\RInd}{RInd}
\def\Res{\operatorname{Res}}

\newtheorem{lemma}{Lemma}[section]

\newtheorem{proposition}[lemma]{Proposition}

\newtheorem{corollary}[lemma]{Corollary}

\newtheorem{lemmas}{Lemma}[subsection]
\newtheorem{lemmadefinitions}[lemmas]{Lemma-Definition}
\newtheorem{propositions}[lemmas]{Proposition}
\newtheorem{theorems}[lemmas]{Theorem}
\newtheorem{corollarys}[lemmas]{Corollary}

\theoremstyle{definition}

\newtheorem{definition}[lemma]{Definition}

\newtheorem{definitions}[lemmas]{Definition}

{

}

\theoremstyle{remark}

\newtheorem{remark}[lemma]{Remark}
\newtheorem{remarks}[lemmas]{Remark}

\newdimen\uboxsep \uboxsep=1ex
\def\uboxn#1{\vtop to 0pt{\hrule height 0pt depth 0pt\vskip\uboxsep
\hbox to 0pt{\hss #1\hss}\vss}}

\def\uboxs#1{\vbox to 0pt{\vss\hbox to 0pt{\hss #1\hss}
\vskip\uboxsep\hrule height 0pt depth 0pt}}

\def\codim{\operatorname{codim}}

\def\cone{\operatorname{cone}}

\def\Sp{\operatorname{Sp}}
\def\Xs{X^{\mathbf{s}}}

\def\pdim{\operatorname{pdim}}

\def\cone{\operatorname{cone}}

\def\ra{\rangle}
\def\la{\langle}
\def\D{\Dscr}
\def \W{\Wscr}

\def \cL{{\mathcal L}}
\def \cW{{\mathcal W}}

\def \cP{{\mathcal P}}

\def\Irr{\operatorname{Irr}}

\definecolor{ruta2}{rgb}{0.409, 0.459, 0.208}
\long\def\spela#1{{\color{ruta2}#1}}

\makeatletter
\def\namedlabel#1#2{\begingroup
    #2%
    \def\@currentlabel{#2}%
    \phantomsection\label{#1}\endgroup
}
\let\oldmarginpar\marginpar
\def\marginpar#1{{\oldmarginpar{\tiny #1}}}
\def\Sym{\operatorname{Sym}}
\title{Semi-orthogonal decompositions of GIT quotient stacks}
\author{\v{S}pela \v{S}penko}
\thanks{The first author is a FWO $[$PEGASUS$]^2$ Marie Sk\l odowska-Curie fellow at the Free University of Brussels
(funded by the European Union Horizon 2020 research and innovation
programme under the Marie Sk\l odowska-Curie grant agreement
No 665501 with the Research Foundation Flanders (FWO)). During part of this work she was also a postdoc with Sue Sierra at the University of Edinburgh. Partly she was supported by L'Or\'eal-UNESCO scholarship ``For women in science''.}
\email[\v{S}pela \v{S}penko]{Spela.Spenko@vub.ac.be}
\address{Departement Wiskunde, Vrije Universiteit Brussel,
Pleinlaan $2$, B-1050 Elsene}
\author{Michel Van den Bergh}
\email[Michel Van den Bergh]{michel.vandenbergh@uhasselt.be}
\address{Departement WNI, Universiteit Hasselt, Universitaire Campus \\
B-3590 Diepenbeek}
\thanks{The second author is a senior researcher at the Research Foundation Flanders (FWO).  While working on this project he was supported by
the FWO grant G0D8616N: ``Hochschild cohomology and deformation theory of triangulated categories''.}
\thanks{Substantial progress on this project was made during visits of the authors to each other's host institutions.
They respectively thank the University of Hasselt and the University of Edinburgh for their hospitality and support.}

\begin{document}
\def\vdb#1{\textcolor{red}{#1}}
\def\spela#1{\textcolor{blue}{#1}}
\begin{abstract}
  If $G$ is a reductive group acting on a linearized smooth
  scheme~$X$ then we show that under suitable standard conditions the
  derived category $\Dscr(X^{ss}/G)$ of the corresponding GIT quotient
  stack $X^{ss}/G$ has a semi-or\-thogonal decomposition consisting of
  derived categories of coherent sheaves of rings on $X^{ss}\quot G$
  which are locally of finite global dimension. One of the
components of the decomposition is a certain non-commutative
  resolution of $X^{ss}\quot G$ constructed earlier by the authors. As a concrete example we obtain in the case of odd Pfaffians a semi-orthogonal decomposition of the corresponding quotient
stack in which all the parts are certain specific non-commutative \emph{crepant} resolutions of Pfaffians of lower or equal rank which had also been constructed earlier by the authors.  In particular this semi-orthogonal decomposition cannot be refined further since its parts are Calabi-Yau.

The results in this paper complement results by Halpern-Leistner, Ballard-Favero-Katzarkov and Donovan-Segal
 that assert the existence of a
semi-orthogonal decomposition of $\Dscr(X/G)$ in which one of the
parts is $\Dscr(X^{ss}/G)$.
\end{abstract}
\maketitle

\section{Introduction}
\subsection{Main result}
\label{sec:intro:mainresult}
Throughout $k$ is an algebraically closed base field of characteristic~$0$. All schemes are $k$-schemes.
 If $\Lambda$ is a right noetherian
ring then we write $\Dscr(\Lambda)$ for $D^b_f(\Lambda)\subset D(\Lambda)$, the bounded derived category
of right $\Lambda$-modules with finitely generated cohomology. Similarly for a noetherian scheme/stack $\Xscr$ we write
$\Dscr(\Xscr):=D^b_{\coh}(\Xscr)$ and if $\Ascr$ is a quasi-coherent sheaf of noetherian algebras on a stack $\Xscr$ then we write $\Dscr(\Ascr)$ for $D^b_{\coh}(\Ascr)$.
\begin{definitions}
Let $\Dscr$ be a triangulated category. A \emph{semi-orthogonal decomposition}
$\Dscr=\la \D_i\mid i\in I\ra$ is a list of triangulated subcategories $(\D_i)_{i\in I}$ of $\Dscr$
indexed by a totally ordered set $I$ such that
\begin{enumerate}
\item $\D$ is \emph{generated} by $\D_{i}$, $i\in I$. I.e. it is
the smallest triangulated subcategory of $\D$ containing $\D_{i}$, $i\in I$. 
\item $\Hom(\D_{i},\D_{j})=0$ for $j<i$.
\end{enumerate}
\end{definitions}
Let $X$ be a scheme. A \emph{presheaf of triangulated categories} $\widetilde{\Escr}$ on $X$ consists
of triangulated categories $\widetilde{\Escr}(U)$ for all open subschemes
 $U\subset X$
together with exact restriction
functors $\widetilde{\Escr(U)}\r \widetilde{\Escr(V)}$ for $V\subset U$ satisfying the usual compatibilities. A triangulated
subpresheaf $\widetilde{\Fscr}$ of $\widetilde{\Escr}$ is a collection of triangulated subcategories
$\widetilde{\Fscr}(U)\subset \widetilde{\Escr}(U)$ compatible with restriction.

A semi-orthogonal decomposition $\widetilde{\Escr}=\la \widetilde{\Escr}_i\mid i\in I\ra$ is a list of triangulated
subpresheaves $(\widetilde{\Escr}_i)_{i\in I}$ of $\widetilde{\Escr}$
indexed by a totally ordered set $I$ such that for each open $U\subset X$
we have a semi-orthogonal decomposition $\widetilde{\Escr}(U)=\langle \widetilde{\Escr}_i(U)\mid i\in I\rangle$.

If $X$ is noetherian then we write  $\widetilde{\Dscr}_X$ for the presheaf of triangulated categories
$U\mapsto \Dscr(U)$ on $X$. For a quasi-coherent sheaf of noetherian algebras $\Ascr$ on $X$ we similarly
put $\widetilde{\Dscr}_\Ascr(U)=\Dscr(\Ascr\mid U)$.

  Let $G$ be a reductive group acting on a $k$-scheme $X$ such
  that a ``good quotient'' $\pi:X\r X\quot G$ exists (see \S\ref{sec:goodq}
  below). Then we define a presheaf of triangulated categories $\widetilde{\Dscr}_{X/G}$ on $X\quot G$
as follows: if $U\subset X\quot G$ is open then we put $\widetilde{\Dscr}_{X/G}(U)=\Dscr((U\times_{X\quot G} X)/G)$.
\begin{theorems}\label{nonl}
  Let $G$ be a reductive group acting on a smooth variety\footnote{Variety here means an integral separated noetherian $k$-scheme.} $X$ such
  that a good quotient $\pi:X\r X\quot G$ exists
 and put $\tilde{\Dscr}=\tilde\Dscr_{X/G}$.
There exist
  one-parameter subgroups $\lambda_i:G_m\r G$ of $G$, open subgroups $\open^{\lambda_i}$
of $G^{\lambda_i}$
and finite dimensional
  $\open^{\lambda_i}$-representations $U_i$, such that $\widetilde{\Dscr}=\la
  \dots ,\widetilde{\Dscr}_{-2},\widetilde{\Dscr}_{-1},\widetilde{\Dscr}_0\ra$ with
  $\widetilde{\D}_{-i}\cong\widetilde{\Dscr}_{\Lambda_i}$ for sheaves of
  $\Oscr_{X^{\lambda_i}\quot \open^{\lambda_i}}$-algebras (viewed as
$\Oscr_{X\quot G}$-algebras) defined by $\Lambda_i\cong
  (\End(U_i)\otimes_k \pi_\ast\Oscr_{X^{\lambda_i}})^{\open^{\lambda_i}}$. The restrictions of $\Lambda_i$ to affine opens
   have finite global dimension.
\end{theorems}
In this theorem the notation $(-)^{\lambda_i}$ was used for the fixed points
under $\lambda_i$ (see \S\ref{sec:goodq} below). Note that $G^{\lambda_i}$ is a reductive subgroup of $G$
(see \S\ref{sec:red}) acting on $X^{\lambda_i}$.

\medskip

Theorem \ref{nonl} applies in particular to GIT stack quotients of the
form $X^{ss}/G$ where~$X$ is a smooth projective variety over an
affine variety equipped with an ample linearization. In that way Theorem
\ref{nonl} complements \cite[Theorem 2.10]{HL} (and similar results in \cite{BFK,SegalDonovan})  which constructs a
semi-orthogonal decomposition of $\Dscr(X/G)$ in which one of the
parts is $\Dscr(X^{ss}/G)$.

\subsection{The linear case}\label{linear}
The proof of Theorem \ref{nonl} will be reduced ultimately to the case
that $G$ is connected and $X$ is a representation.
In this section we give a more precise description of the
semi-orthogonal decomposition in this case.

We first need to
introduce  more notation.  Let $T\subset B$ be a maximal torus and
a Borel subgroup in $G$.  Let $X(T)$ and $Y(T)=X(T)^\vee$ be
respectively the character group of $T$ and the group of one-parameter
subgroups of $T$.  Let the roots of $B$ be the negative roots and let
$X(T)^{\pm}$, $Y(T)^{\pm}$ be the \hbox{(anti-)}dominant cones in $X(T)$ and $Y(T)$.
Let $\bar{\rho}\in X(T)_\RR$ be half the sum of the positive roots.

Let $W$ be a finite dimensional $G$-representation
of dimension $d$ such that $X=W^\vee$ and let $R=k[X]=\Sym(W)$.  Let
$(\beta_i)_{i=1}^d\in X(T)$ be the $T$-weights of $W$. For $\lambda\in
Y(T)^-$  define the following subsets of $X(T)_\RR$
\begin{align*}
\Sigma_\lambda&=\left\{\sum_i a_i \beta_i\mid a_i\in ]-1,0],\, \la\lambda,\beta_i\ra=0\right\},\qquad
\Sigma:=\Sigma_0.
\end{align*}
We denote $\Sigma_\lambda^0={\rm relint}\Sigma_\lambda=\{\sum_i a_i
\beta_i\mid a_i\in ]-1,0[,\, \la\lambda,\beta_i\ra=0\}$.  With $W_\lambda$ we denote the coinvariants for the action of $\lambda$ (i.e. the quotient
space of $W$ obtained by dividing out the weight vectors $w_i$ such that $\la\lambda,\beta_i\ra\neq
0$).  We further denote (see also \S\ref{sec:red} below):
\begin{align*}
G^{\lambda,+}&=\{g\in G\mid \lim_{t\r 0} \lambda(t)g\lambda(t)^{-1}\text{ exists }\},
\end{align*}
(Note that $G^\lambda=G^{\lambda,+}/\mathrm {rad}\,G^{\lambda,+}$ is
the reductive Levi factor of $G^{\lambda,+}$ containing $T$.)  Let
$\cW=N(T)/T$ be the Weyl group of $G$, and let $\cW_{G^\lambda}\subset
\cW$ be the Weyl group of~$G^\lambda$.  We write
$\bar\rho_\lambda\in X(T)_\RR$ for  half the sum the  positive roots of $G^\lambda$
and $X(T)^\lambda$ for the
$G^\lambda$-dominant weights inside $X(T)$.  For $\chi\in X(T)^\lambda$ we write
$V_{G^\lambda}(\chi)=\Ind^{G^\lambda}_{G^\lambda\cap B} \chi$. I.e $V_{G^\lambda}(\lambda)$ is the irreducible $G^\lambda$-representation with highest weight
$\chi$ (or sometimes also a $G^{\lambda,+}$-representation with the
unipotent radical acting trivially).  Note that $G^\lambda$ acts on
$W_\lambda$ (when we consider it as $G^{\lambda,+}$ representation we
let $\rm{rad}\, G^{\lambda,+}$ act trivially). For a
$\cW_{G^\lambda}$-invariant $\nu\in X(T)_\RR$ we put
 \begin{align*}
{\Lscr_{r,\lambda,\nu}}&=X(T)^\lambda\cap (\nu-\bar{\rho}_{\lambda}+r\Sigma_\lambda^0),\\
U_{r,\lambda,\nu}&=\bigoplus_{\mu\in \Lscr_{r,\lambda,\nu}} V_{G^\lambda}(\mu),\\
\Lambda_{r,\lambda,\nu}&=(\End(U_{r,\lambda,\nu})\otimes_k \Sym(W_\lambda))^{G^\lambda}.
\end{align*}

\begin{propositions}\label{sigma}\cite[Theorem 1.4.1]{SVdB} Assume $r\ge 1$. Then
one has $\gldim \Lambda_{r,\lambda,\nu}<\infty$.
\end{propositions}
\begin{proof}
  As $\nu$ is $\cW_{G^\lambda}$-invariant, this
  follows from \cite[Theorem 1.4.1]{SVdB} in the case
  $\Sigma_\lambda^0=\Sigma_\lambda$.
  Let $\Gamma=-\bar{\rho}+\left\{\sum_i a_i \beta_i\mid a_i\le 0\right\}$, $\Gamma^0={\rm relint \Gamma}$.
   It is easy to modify the proof
  to hold also if $\Sigma_\lambda^0\subsetneq \Sigma_\lambda$.  To
  replace $\Sigma$ in \cite[Theorem 1.4.1]{SVdB} with $\Sigma^0$ one
  only needs to show that $\Hom(P_{\cL},P_\chi)=0$ if $\chi\not \in
  \Gamma^0$, and this follows by the same proof as \cite[Lemma
  11.3.1]{SVdB}.  To replace $\Sigma_\lambda^0$ by
  $r\Sigma_\lambda^0$ is also easy.
\end{proof}
We say that $x\in X$ is $T$-stable if $x$ has finite stabilizer and closed $T$-orbit. In the case $X=\Spec W$
the existence of a $T$-stable point is equivalent to the cone spanned by the weights $(\beta_i)_i$ of $W$ being equal ot $X(T)_\RR$.
\begin{propositions}\label{ref-1.3}
Let $\Dscr=\Dscr(X/G)$ and assume that $X$ has a $T$-stable point. Then there exist $r_i\ge 1$, $\lambda_i\in Y(T)^-$
and  $\cW_{G^{\lambda_i}}$-invariant $\nu_i\in X(T)_\RR$
such that
$\Dscr=\la \dots ,\Dscr_{-2},\Dscr_{-1},\Dscr_0\ra$,
$\D_{-i}\cong\Dscr(\Lambda_{r_i,\lambda_i,\nu_i})$,
 is a semi-orthogonal decomposition of $\D$. Moreover we may assume $\Dscr_0\cong\Dscr(\Lambda_{1,0,0})$.
\end{propositions}
\begin{remarks} The reader will note that (under suitable genericity conditions) $\Lambda_{1,0,0}$ is the non-commutative
resolution of $X\quot G$ constructed in \cite[Cor.\ 1.5.2]{SVdB}.
\end{remarks}
\begin{remarks}
\label{rem:notTstable}
In Proposition \ref{ref-1.3} we assume that
$X$ has a $T$-stable point. This hypothesis is not
very restrictive in view of \S\ref{sec:nonstable} below.
Roughly speaking if $X$ does not have a $T$-stable point then one may easily obtain a semi-orthogonal decomposition of $\Dscr(X/G)$ involving a set of $\Dscr(X'/G')$  such that $X'$ has a $T'$-stable point
for $T'$ a maximal torus of $G'$.
\end{remarks}
\subsection{Refined decompositions}
The semi-orthogonal decompositions in Theorem \ref{nonl} and Proposition \ref{sigma} are not optimal.
At the cost of extra technicalities one may decompose the parts $\Dscr_{-i}$ further by essentially repeating
the procedure which is used to obtain the decomposition of $\Dscr$ itself. This leads to the natural
problem to produce a decomposition which cannot be refined further in the sense that the parts admit no  non-trivial
semi-orthogonal decompositions. The latter is in particular the case if the parts are all of the form $\Dscr(\Lambda)$
where  $\Lambda$ is a  (possibly twisted) \emph{non-commutative crepant resolution (NCCR)} of its center \cite{Leuschke,VdB32,Wemyss1}.\footnote{The indecomposability of (twisted) NCCR's seems to be well known to experts. It may
be easily proved in a similar way as \cite[Lemma A.4]{SVdB5}.}
In \S\ref{appA} we will discuss this problem in the case that $W$ is ``quasi-symmetric'', i.e. the sum of the weights
  of $W$ on each line through the origin is zero.

It is shown in
  \cite[\S1.6]{SVdB} that in this case, by replacing $\Sigma$ by a polygon of
  roughly half the size, one obtains smaller non-commutative
  resolutions for $X\quot G$. Under favourable conditions one may even
  obtain NCCRs.\footnote{See also \cite{HLSam} where, again under appropriate
  conditions, it is shown that these NCCRs are of geometric origin in
  the sense that they are derived equivalent to suitable $X^{ss}/G$.}

Likewise in \S\ref{appA} we will show that if~$W$ is quasi-symmetric one may obtain corresponding 
  more refined semi-orthogonal decomposition of $\Dscr(X/G$) which again under favorable conditions
consists entirely of (twisted) NCCR parts. So they cannot be refined further.  See  \S\ref{appA}, in particular Corollary \ref{quasicor}.
As an example we mention the following explicit result which refines our construction of NCCRs for odd Pfaffians in \cite{SVdB}.
\begin{proposition}[see Proposition \ref{prop:pfaffians}]
\label{prop:intro}
Let $2n<h$, $W=V^h$, where $V$ is $2n$-dimensional vector space equipped with a non-degenerate skew-symmetric bilinear form, $G={\rm Sp}_{2n}(k)$, and let $Y_{2n,h}^-=W\quot G$ be the variety of skew-symmetric $h\times h$-matrices of rank $\leq 2n$. 
We denote by $\Lambda_j$ the NCCR of $Y_{2j,h}$ given by \cite[Proposition 6.1.2]{SVdB} (by convention we put $\Lambda_0=k$).
If $h$ is odd then $\Dscr(X/G)$ has a semi-orthogonal decomposition $\langle \dots, \D_{-2},\D_{-1},\D_0\ra$ with $\D_0\cong \D(\Lambda_n)$ such that each $\D_{-i}$ for $i>0$ is of the form $\D(\Lambda_j)$ for some $j<n$. 
\end{proposition}
Other examples we discuss are quasi-symmetric toric representations, representations of $\rm SL_2$ and the analogue of Proposition \ref{prop:intro} for ordinary determinantal varieties.

\section{Acknowledgement}
The second author thanks J\o rgen Vold Rennemo and Ed Segal for interesting discussions regarding this paper. The authors thank  Agnieszka Bodzenta and Alexey Bondal for their interest in this work and
for useful comments on the first version of this paper.
\section{Preliminaries}
\subsection{Strongly \'etale morphisms}
Let $G$ be reductive group. If $G$ acts on an affine~$k$-scheme $X$ then we put $X\quot G=\Spec k[X]^G$. This
is a special case of a ``good quotient'' (see \S\ref{sec:goodq} below).
Let  $f:X\r Y$ be a $G$-equivariant morphism between affine $G$-schemes. Following \cite{Luna}\cite[App.\ D]{Mumford}  we
say that $f$ is \emph{strongly \'etale} if $X\quot G\r Y\quot G$ is \'etale and the induced morphism $X\r Y\times_{Y\quot G} X\quot G$
is an isomorphism. This implies in particular that $X\r Y$ is \'etale, which is a special case of the following lemma with $H$ being trivial.
\begin{lemmas} Assume that $f:X\r Y$ is a strongly \'etale $G$-equivariant morphism of affine
schemes and let $H$ be a reductive subgroup of $G$. Then
$f$ is strongly \'etale as $H$-equivariant morphism.
\end{lemmas}
\begin{proof} From the fact that $H$ is reductive we easily obtain
\[
k[X]^H=(k[Y]\otimes_{k[Y]^G} k[X]^G)^H=k[Y]^H\otimes_{k[Y]^G} k[X]^G.
\]
Thus
\[
k[Y]\otimes_{k[Y]^H} k[X]^H=k[Y]\otimes_{k[Y]^H}k[Y]^H\otimes_{k[Y]^G} k[X]^G=k[Y]\otimes_{k[Y]^G} k[X]^G=k[X].
\]
\end{proof}
\subsection{The Bia\l{}ynicki-Birula decomposition in the affine case.}
\label{ref-1.2-0}
We  use
 \cite{Drinfeld2} as a reference for some facts about the Bia\l{}ynicki-Birula decomposition~\cite{Bia}.
Let $R$ be a commutative $k$-algebra equipped with a rational $G_m$-action $\lambda:G_m\r \Aut_k(R)$. This $G_m$-action
induces a grading on $\bigoplus_n R_n$ on $R$ where $z\in G_m$ acts on $r\in R_n$ by
$z\cdot r=z^n r$. Let $I^+$, $I^-$ be the ideals in $R$ respectively generated by $(R_n)_{n>0}$ and $(R_n)_{n<0}$
and put $I=I^++I^-$. We define $R^\lambda:=R/I$. $R^{\lambda,\pm}:=R/I^{\pm}$. Note
\begin{equation}
\label{eq:part0}
(R^{\lambda,\pm})_0=R^\lambda.
\end{equation}
If $X=\Spec R$ then we also write $X^\lambda=\Spec R^\lambda$, $X^{\lambda,\pm}=\Spec R^{\lambda,\pm}$. It follows from
\cite[\S1.3.4]{Drinfeld2} that $X^\lambda$ is the subscheme of fixed points of $X$ and $X^{\lambda,+}$, $X^{\lambda,-}$
are respectively the attractor and repeller subschemes of $X$.
According to \cite[Prop 1.4.20]{Drinfeld2}, $X^\lambda$, $X^{\lambda,\pm}$ are smooth if this is the case
for $X$.
\begin{lemmas}
\label{ref-1.2.1-1} Assume that $f:X\r Y$ is a strongly \'etale
  $G_m$-equivariant morphism of affine schemes, with the action denoted by $\lambda$. Then
\begin{align*}
X^\lambda&= X\times_{Y} Y^\lambda\\
X^{\lambda,+}&= X \times_{Y} Y^{\lambda,+}\\
X^{\lambda,-}&= X\times_{Y} Y^{\lambda,-}\,.
\end{align*}
\end{lemmas}
\begin{proof} We have
\[
k[X]= k[Y]\otimes_{k[Y]^{G_m}} k[X]^{G_m}
\]
since $f$ is a strongly \'etale
  $G_m$-equivariant morphism,
and this isomorphism is clearly compatible with the grading on both sides. Thus
\[
k[X]_n=k[Y]_n\otimes_{k[Y]_0} k[X]_{0}.
\]
The lemma now follows easily from the definitions.
\end{proof}
\subsection{The Bia\l{}ynicki-Birula decomposition when there is a good quotient.}
\label{sec:goodq}
We use the following definition from \cite{Brion} (see the discussion after Prop.\ 1.29 in loc.\ cit.).
\begin{definitions}
Let $G$ be a reductive group and let $\pi:X\r Y$ be a $G$-equivariant morphism of $k$-schemes. Then
$\pi$ is a \emph{good quotient} if the following holds
\begin{enumerate}
\item $\pi$ is affine.
\item $\Oscr_Y=(\pi_\ast \Oscr_X)^G$.
\end{enumerate}
\end{definitions}
It is easy to see that a good quotient is unique, if it
exists. Therefore following tradition we will usually write $Y=X\quot
G$. We have already used this notation in the case that $X$ is affine.
 Note the following
\begin{lemmas} Assume that $G$ is a reductive group acting on a $k$-scheme $X$ such that $X\quot G$ exists.
Let $H$ be a reductive subgroup of $G$. Then $X\quot H$ also exists.
\end{lemmas}
\begin{proof} Let $\pi:X\r X\quot G$ be the good quotient. It is easy to verify that $X\quot H=\underline{\Spec} (\pi_\ast \Oscr_X)^H$.
\end{proof}
Assume now that $X$ is a $k$-scheme on which $G_m$ acts via
$\lambda:G_m\r \Aut(X)$. Assume that a good quotient $\pi:X\r X\quot
G_m$ exists. If $U\subset X\quot G_m$ is an open affine subvariety
then we may define closed subvarieties $\pi^{-1}(U)^\lambda$,
$\pi^{-1}(U)^{\lambda,\pm}$ of $\pi^{-1}(U)$ as in \S\ref{ref-1.2-0}
and according to Lemma
\ref{ref-1.2.1-1} these are compatible with restrictions for
$U'\subset U$. Hence we may glue these closed subvarieties to obtain
$X^\lambda$, $X^{\lambda,\pm}\subset X$. One may verify that $X^\lambda$, $X^{\lambda,\pm}$ are still
the fixed points and the attractor/repeller subschemes for~$\lambda$.

\subsection{Good quotients and geometric invariant theory}
One way to obtain good quotients is via the machinery of geometric
invariant theory \cite{Mumford}. Let~$G$ be a reductive group and let $X$ be
a~$G$-equivariant $k$-scheme which is projective over an affine scheme, equipped with a~$G$-equivariant ample line bundle $\Mscr$. If $f\in
\Gamma(X,\Mscr^{\otimes_k n})^G$, $n>0$, then $X_f:=\{f\neq 0\}\subset X$ is
affine and $G$-equivariant. The semi-stable locus in~$X$ is defined as $X^{ss}=\bigcup_f X_f$. This is an open
subvariety of~$X$ which has a good quotient $X^{ss}\quot G$ which may
be obtained by gluing $X_f\quot G=\Spec k[X_f]^G$ for varying $f$.
Another way to
obtain $X^{ss}\quot G$ is as follows: let $\Gamma_\ast(X)=\bigoplus_n
\Gamma(X,\Mscr^{\otimes_k n})$. Then $X^{ss}\quot G=\Proj \Gamma_\ast(X)^G$.

\medskip

The following result, whose proof we omit since we do not use it, gives an alternative description of $X^{ss,\lambda}$,
$X^{ss,\lambda,\pm}$ in the GIT setting.
\begin{propositions}
Let $G$ be a reductive group and let
$X$ be a $G$-equivariant $k$-scheme which is projective  over an affine scheme. Let $\Mscr\in \Pic(X)$ be a
$G$-equivariant ample line bundle on $X$ and let $X^{ss}\subset X$ be
the corresponding semi-stable locus.

Let $R=\Gamma_\ast(X)$ and let $\lambda$ be a one-parameter subgroup of $G$. Then $\lambda$ acts
on~$R$ in a way which is compatible with the grading
and  $X^{ss,\lambda}$, $X^{ss,\lambda,\pm}$ are the closed subschemes
of $X^{ss}$ defined by the (graded) quotient rings $R^\lambda$, $R^{\lambda,\pm}$ of $R$ (see \S\ref{ref-1.2-0}).
\end{propositions}
\subsection{Good quotients and local generation}
\label{sec:localgen}
Let $X/k$ be a quasi-compact, quasi-separated $G$-scheme for a reductive
group $G/k$ such that a good quotient $\pi:X\r X\quot G$ exists. It is
easy to see that then $X\quot G$ is quasi-compact and quasi-separated as
well. Below we write $\pi_s:X/G\r X\quot G$ for the corresponding
stack morphism. Note that both $\pi_\ast$ and $\pi_{s\ast}$ are exact.

Below we use some notations and concepts related to derived categories which were introduced in \S\ref{sec:intro:mainresult}.
We recall
some properties of $D_{\Qch}(X/G)$.
\begin{theorems}
\label{th:st}
\begin{enumerate}
\item \label{c1} $D_{\Qch}(X/G)$ is compactly generated.
\item \label{c2} An object in $D_{\Qch}(X/G)$ is compact if and only if it is perfect. I.e.\ if and only if its
image in $D(X)$ is perfect.
\item \label{c3} If $X$ is separated then $D_{\Qch}(X/G)=D(\Qch(X/G))$.
\end{enumerate}
\end{theorems}
\begin{proof}
\eqref{c1} follows from \cite[Thm B]{HallRydh}. It also follows from this result and the proof of \cite[Lemma 2.2]{Neeman3}
that every compact object is perfect. On the other hand it is easy to see that in this case every perfect object
is compact. This proves~\eqref{c2}. Finally \eqref{c3} follows from \cite[Thm 1.2]{HallNeemanRydh}.
\end{proof}
For
an open $U\subset X\quot G$ we write $\tilde{U}=U\times_{X\quot G} X\subset X$.
\begin{definitions}
\label{def:locgen} Let $(E_i)_{i\in I}$
be a collection of perfect objects in $D_{\Qch}(X/G)$. We say that the full subcategory of $D_{\Qch}(X/G)$, spanned by all objects $\Fscr$ such that for every affine open
$U\subset X\quot G$
the object $\Fscr{|}\tilde{U}$ is in the smallest thick subcategory
of $D_{\Qch}(\tilde{U}/G)$ containing $(E_i{|}\tilde{U})_i$, is \emph{locally classically generated}
by $(E_i)_{i\in I}$.
\end{definitions}
Let us say that $F,G\in D_{\Qch}(X/G)$ are \emph{locally isomorphic} if there exists a covering $X\quot G=\bigcup_{i\in I} U_i$
such that $F{\mid} \tilde{U}_i\cong G{\mid} \tilde{U}_i$ for all $i$. It is convenient to call a subcategory
of $D_{\Qch}(X/G)$ \emph{locally closed} if it is closed under local isomorphism.
\begin{lemmas} Let $(E_i)_{i\in I}$ be a collection of perfect objects
  in $D_{\Qch}(X/G)$ and let $\Fscr\in \Perf(X/G)$.  Let $X\quot
  G=\bigcup_{j=1}^n U_j$ be a finite open affine covering of $X\quot G$. If
  for all $j$ one has that $\Fscr{\mid}\tilde{U}_j$ is in the smallest
  thick subcategory of $D_{\Qch}(\tilde{U}_j/G)$ containing
  $(E_i{|}\tilde{U}_j)_i$ then $\Fscr$ is in the subcategory of $D_{\Qch}(X/G)$
locally classically generated
by $(E_i)_{i\in I}$.
\end{lemmas}
\begin{proof} Let $U\subset X\quot G$ be an affine open. We have to show that
$\Fscr{\mid}\tilde{U}$ is in the smallest
  thick subcategory of $D_{\Qch}(\tilde{U}/G)$ containing  $(E_i{\mid} \tilde{U})_{i\in I}$.
 By
replacing $X\quot G$ by $U$ and refining the cover $U=\bigcup_{j\in I} U\cap U_j$ to an affine
one we reduce to the case that $X\quot G$ is itself affine. In particular by Theorem \ref{th:st}\eqref{c3} (as affine schemes are separated),
$D_{\Qch}(X/G)$ is the derived category of $G$-equivariant $k[X]$-modules. The affine $U_j$
yield $G$-equivariant flat extensions $k[\tilde{U}_j]$ of $k[X]$.

Let $\Escr$ be the smallest
cocomplete triangulated subcategory of $D_{\Qch}(X/G)$ 
containing $(E_i)_{i\in I}$
and similarly let $\Escr_j$ be the smallest cocomplete triangulated subcategory of $D_{\Qch}(\tilde{U}_j/G)$ containing
 $(E_i{|} \tilde{U}_j)_{i\in I}$.
By the Brown representability theorem \cite[Theorem 4.1]{Neeman}
there is a unique distinguished triangle
\[
\Fscr_0\r \Fscr\r \Fscr_1\r
\]
where $\Fscr_0\in \Escr$ and $\Fscr_1\in \Escr^\perp$. Since $U_j$ is affine it is easy to see
that $\Fscr_1{\mid} \tilde{U}_j\in \Escr_j^\perp$.
But by hypothesis $\Fscr{\mid} \tilde{U}_j\in \Escr_j$ and thus $\Fscr_1{\mid} \tilde{U}_j=0$.
Since this is true for all $j$ we conclude $\Fscr_1=0$ and hence $\Fscr\in \Escr$.
Since $\Fscr$ is compact and $\Escr$ is compactly generated by Theorem \ref{th:st}\eqref{c1} the conclusion follows from \cite[Lemma 2.2]{Neeman3}.
\end{proof}
\begin{lemmas}
\label{lem:thick}
The category $\Perf(X/G)$  is locally classically generated by $(V\otimes_k \Oscr_X)_V$
where $V$ runs through the irreducible representations of $G$.
\end{lemmas}
\begin{proof}
We may assume that $X$ is affine and then it is clear.
\end{proof}
It will be convenient to pick for every $E\in D_{\Qch}(X/G)$ a $K$-injective resolution
$E\r I_E$ and to define $\pi_{s\ast}\uRHom_{X/G}(E,F)$ as the complex of sheaves
$U\mapsto \Hom_{\tilde{U}}(I_E{\mid}\tilde{U},I_F{\mid}\tilde{U})^G$ on $X\quot G$.
With this definition $\Lambda:=\pi_{s\ast}\uREnd_{X/G}(E):=\pi_{s\ast} \uRHom_{X/G}(E,E)$ is a sheaf of DG-algebras
on $X\quot G$
and $\pi_{s\ast}\uRHom_{X/G}(E,F)$ is a sheaf of right $\Lambda$-DG-modules.
\begin{remarks}
Note that if $U\subset X\quot G$ is affine then $\Lambda{\mid} U$ is the sheaf of DG-algebras
associated to the DG-algebra $\REnd_{\tilde{U}/G}(E)$. We will use this routinely below.
\end{remarks}

\begin{lemmas} \label{lem:ff}
Assume that $\Dscr\subset D_{\Qch}(X/G)$ is locally classically generated by the perfect complex $E$
and let $\Lambda=\pi_{s\ast}\uREnd_{X/G}(E)$ be the sheaf of DG-algebras on $X\quot G$ as defined above.
The functors
\begin{align*}
\Dscr\r \Perf(\Lambda)&:F\mapsto \pi_{s\ast}\uRHom_{X/G}(E,F),\\
\Perf(\Lambda)\r \Dscr&:H\mapsto H\Lotimes_\Lambda E,
\end{align*}
are well-defined (the second functor is computed starting from a $K$-flat resolution\footnote{Such a $K$-flat resolution is constructed in the same way as for DG-algebras
(see \cite[Theorem 3.1.b]{Keller1}). One starts from the observation that for every $M\in D(\Lambda)$ there is a morphism $\bigoplus_{i\in I} j_{i!}(\Lambda{\mid} U_i)\r M$
with open immersions $(j_i:U_i\r X\quot G)_{i\in I}$, which is an epimorphism on the level of cohomology.
 }
 of $H$) and yield inverse equivalences between $\Dscr$ and $\Perf(\Lambda)$.
\end{lemmas}
\begin{proof} The two functors are  adjoint functors between $D_{\Qch}(X/G)$ and $D(\Lambda)$. The fact that they define functors between
$\Dscr$ and $\Perf(\Lambda)$ can be checked locally. The fact  that the unit and counit are invertible can also be checked locally.
\end{proof}
\begin{lemmas} \label{gldimadm} Assume that $X$ is a smooth
  $k$-scheme. Let $E\in \Dscr(X/G)$. If $\Lambda=\pi_{s\ast}\uREnd_{X/G}(E)$ is a
  sheaf of algebras of finite global dimension when restricted to
   affine opens in $X\quot G$ then the induced fully faithful
  functor (see Lemma \ref{lem:ff})
\[
I:\Dscr(\Lambda)\r \Dscr(X/G):H\mapsto H\Lotimes_\Lambda E
\]
is admissible (i.e.\ it has a left and a right adjoint).
\end{lemmas}
\begin{proof} The right adjoint to $I$ is given $\pi_{s\ast}\uRHom_{X/G}(E,-)$. To construct the left adjoint note that there is a duality
$\Dscr(\Lambda^\circ)\r \Dscr(\Lambda)$ given by $(-)^\vee:=\uRHom_{\Lambda}(-,\Lambda)$. One checks that the
left adjoint to $I$ is given by $\pi_{s\ast}\uRHom_{X/G}(-,E)^\vee$.
\end{proof}
The following result shows that semi-orthogonal decompositions can be constructed locally.
\begin{propositions}
\label{th:recognition}
Let $I$ be a  totally ordered set. Assume
$\Dscr\subset\Perf(X/G)$ is locally classically generated by
a collection of locally closed subcategories $\Dscr_i\in \Perf(X/G)$.
Assume $\pi_{s\ast}\uRHom_{X/G}(\Dscr_i,\Dscr_j)=0$ for $i>j$. Then $\Dscr$ is generated by $(\Dscr_i)_i$
and in particular we have a semi-orthogonal decomposition $\Dscr=\langle \Dscr_i\mid i\in I\rangle$.
\end{propositions}
\begin{proof} It is clear that we may first reduce to the case that $I$ finite and then to $|I|=2$. Hence
we assume $I=\{1,2\}$. In the same vein we may reduce to the case that the $\Dscr_i$ are locally classically generated by single perfect complexes $(E_i)_{i=1,2}$. Put $\Lambda_i=\pi_{s\ast} \uREnd_{X/G}(E_i)$.

Let $F\in \Dscr$. Then for every affine $U\subset X\quot G$,
$F{\mid} \tilde{U}$ is in the thick subcategory of $\Perf(\tilde{U}/G)$ generated by $E_1{\mid} \tilde{U}$, $E_2{\mid} \tilde{U}$ (by the definition of local classical generation, cfr.
Definition \ref{def:locgen}).

Put $F_2=\pi_{s\ast} \uRHom_{X/G}(E_2,F)\Lotimes_{\Lambda_2} E_2$. 
Since $ \uRHom_{X/G}(E_2,E_1)=0$ we deduce (checking locally) that $\pi_{s\ast} \uRHom_{X/G}(E_2,F)\in \Perf(\Lambda_2)$ and hence that $F_2\in \Dscr_2$.
Put
$F_1=\cone(F_2\r F)$.
Then we obtain $\pi_{s\ast}\uRHom_{X/G}(E_2,F_1)=0$ (again checking locally). Let $C$ be the cone of $\pi_{s\ast} \uRHom_{X/G}(E_1,F_1)\Lotimes_{\Lambda_1} E_1\r F_1$.
Since $\pi_{s\ast} \uRHom_{X/G}(E_i,C)=0$ for $i=1,2$ we conclude $C=0$ and thus $F_1=\pi_{s\ast} \uRHom_{X/G}(E_1,F_1)\Lotimes_{\Lambda_1} E_1$.
It follows that if $U\subset X\quot G$ is affine
then $F_1{\mid} \tilde{U}$ is in the cocomplete subcategory of $D_{\Qch}(X/G)$ generated by $E_1{\mid} \tilde{U}$. Since $E_1{|}\tilde{U}$, $F_1{|}\tilde{U}$ are compact
we conclude  by \cite[Lemma 2.2]{Neeman3} that
$F_1{\mid} \tilde{U}$ is in the thick subcategory of $\Perf(\tilde{U}/G)$ generated by $E_1{\mid} \tilde{U}$. As this is is true for all $U$ we obtain that $F_1\in \Dscr_1$. 
Hence
$\Dscr$ is generated by $\Dscr_1$,~$\Dscr_2$.
\end{proof}
\subsection{The Bia\l{}ynicki-Birula decomposition for reductive algebraic groups.}
\label{sec:red}
We recall the following.
\begin{propositions} \cite[Proposition 8.4.5, Exercise 8.4.6(5), Theorem 13.4.2]{Springer}
  Let $G$ be a connected reductive algebraic group and let
  $\lambda:G_m\r G$ be a one-parameter subgroup of $G$. Then $G^{\lambda}$, $G^{\lambda,\pm}$ are connected
subgroups of $G$. Moreover the $G^{\lambda,\pm}$ are parabolic subgroups of $G$ and $G^\lambda$ is the
Levi-subgroup of $G^{\lambda,\pm}$.
\end{propositions}
We recall the following.
\begin{lemmas}
\label{lem:recall} Let $G$ be a connected reductive algebraic group with $T\subset B\subset G$ being
a maximal torus and a Borel subgroup of $G$. Let
  $\lambda\in Y(T)^-$  and $\chi\in X(T)^+$ and let $V(\chi)$ be the irreducible $G$-representation with
highest weight $\chi$. Then $\Res^G_{G^\lambda} V(\chi)=V_{G^\lambda}(\chi)\oplus
\bigoplus_i V_{G^{\lambda}}(\mu_i)$
with $\langle \lambda,\mu_i\rangle>\langle \lambda,\chi\rangle$.
\end{lemmas}
\begin{proof} This is similar to the proof that $\chi$ occurs with multiplicity one
among the weights of $V(\chi)$ \cite[Proposition 2.4]{Jantzen}.
All the weights $\mu$ of $V(\chi)$ satisfy $\langle
  \lambda,\mu\rangle \ge \langle \lambda ,\chi\rangle$ as $\lambda\in Y(T)^-$. Hence we have
  a decomposition $V(\chi)=V(\chi)^\lambda\oplus V(\chi)^+$ where
  $V(\chi)^+$ is the span of the weight vectors with weights $\mu$
  such that $\langle \lambda,\mu\rangle > \langle \lambda
  ,\chi\rangle$. It is clear that this is a decomposition as
  $G^\lambda$-modules. If $V(\chi)^\lambda$ is decomposable then it is
  easy to see that its indecomposable summands generate distinct
  $G$-subrepresentations of $V(\chi)$ which is impossible.

Since $V(\chi)^\lambda$ contains the weight vector with weight $\chi$ we must have $V(\chi)^\lambda=
V_{G^\lambda}(\chi)$.
\end{proof}
\subsection{The $G/G_e$-action on weights}
\label{sec:definition}
Let $G$ be a reductive group such that $T\subset B\subset G_e$ are respectively a
maximal torus and a Borel subgroup of $G_e$.

Let $g\in G$ and $\sigma_g=g\,\cdot\,g^{-1}\in \Aut(G_e)$.  Then
$\sigma_g(T)\subset\sigma_g(B)$ are respectively a maximal torus and a
Borel subgroup of $G_e$. Thus there exists $g_0\in G_e$ such that
$g_0\sigma_g(T)g_0^{-1}=T$, $g_0\sigma_g(B)g_0^{-1}=B$.

In the sequel if $\bar{g}\in G/G_e$ then we write $\sigma_{\bar{g}}\in
\Aut(G_e)$ for $\sigma_{g_0g}$ where $g_0g$ is an element of the coset
$\bar{g}$ such that $\sigma_{g_0g}$ preserves $(T,B)$. Since $g_0$ is
unique up to multiplication by an element of $T$,
$\sigma_{\bar{g}}$ is well defined up to conjugation by an element
of $T$.  Since $\sigma_{\bar{g}}$ preserves $(T,B)$ it yields a
well defined action on $X(T)$ via $\chi\mapsto \chi\circ \sigma_{\bar{g}}$
which preserves $X(T)^+$. We will write
$\bar{g}(\chi)$ for $\chi\circ \sigma^{-1}_{\bar{g}}$. There is
also an action of $G/G_e$ on $Y(T)$ given by $\lambda\mapsto \sigma_{\bar{g}}\circ \lambda$.
Finally we have
\[
\langle \lambda,\chi\circ \sigma_{\bar{g}}\rangle=\langle \sigma_{\bar{g}}\circ\lambda,\chi\rangle.
\]
If $\lambda\in Y(T)$ then we will write $(G/G_e)^\lambda\subset G/G_e$ for the
stabilizer of $\lambda$ under the $G/G_e$-action on $Y(T)$. Let $\tilde{G}^\lambda$ be the inverse image
of $(G/G_e)^\lambda$ in $G$.
There is an obvious
inclusion $(G/G_e)^\lambda \subset G_eG^\lambda/G_e=G^\lambda/G^\lambda_e$. We will
write $\open^\lambda$ for the (open) subgroup of $G^\lambda$ such that
$G^\lambda_e\subset \open^\lambda$ and $\open^\lambda/G^\lambda_e=(G/G_e)^\lambda $.
So $\tilde{G}^\lambda/G_e=\open^\lambda/G_e^\lambda$.

For $\chi\in X(T)^+$ we put
$V_G(\chi):= \Ind^G_{B}\chi$. Note that if $G$ is not connected then
$V_G(\chi)$ will usually not be simple.
We have
\begin{equation}
\label{eq:restriction}
\Res^G_{G_e} V_G(\chi)=\bigoplus_{\bar{g}\in G/G_e} {}_{\sigma_{\bar{g}}} V_{G_e}(\chi)
\end{equation}
and
\begin{equation}
\label{eq:restriction1}
{}_{\sigma_{\bar{g}}}V_{G_e}(\chi)\cong V_{G_e}(\chi\circ \sigma_{\bar{g}}).
\end{equation}
\section{Reduction settings}\label{secredset}
Now we introduce our main technical tool to obtain semi-orthogonal decompositions of $\Dscr(X/G)$.
In \S\ref{def:reductionsetting} introduce the concept of a \emph{reduction setting}.
In \S\ref{sec:mainred} we give our main technical result about such reduction settings.
Subsequent sections are concerned with the construction of reduction settings. Since the definitions
and results are quite technical and not so easy to motivate, the reader is advised to skim this section
on first reading and come back to it afterwards.
\subsection{Definition}
\label{def:reductionsetting}

Let $G$ be a reductive group such that $T\subset B\subset G_e$ are respectively a
maximal torus and a Borel subgroup of $G_e$

\medskip

Below we consider the
situation where $G$ acts on a variety $X$. In that case we also put for $\chi\in X(T)^+$
\[
P_\chi=V_G(\chi)\otimes_k \Oscr_X \in \coh(X/G).
\]
To indicate context we may also write $P_{G,\chi}$, $P_{G,X,\chi}$, etc\dots.
If $\Lscr\subset X(T)^+$
then we put $P_\Lscr=\bigoplus_{\chi\in \Lscr} P_\chi$.
We make the following definition (using some notation introduced in \S\ref{linear}).
\begin{definitions}
\label{ref-2.1-2}
A \emph{reduction setting}
is a tuple
$(G,B,T,X,\Lscr,\chi,\lambda)$ with the following properties:
\begin{enumerate}
\item \label{ref-1-3}
$G$ is a reductive group and $T\subset B\subset G_e$ are respectively a maximal
torus and a Borel subgroup of $G$.
\item $\chi\in X(T)^+$.
\item $\lambda\in Y(T)^-$.
\item $\Lscr$ is a finite subset of $X(T)^+$
invariant under $G/G_e$.
\item \label{ref-5-4}
$\forall \mu\in \Lscr:\langle \lambda,\chi\rangle<\langle\lambda,\mu\rangle$.
\item \label{ref-6-5} $X$ is a smooth $G$-equivariant $k$-scheme such
  that a good quotient $\pi:X\r X\quot G$ exists (with associated stack morphism $\pi_s:X/G\r X\quot G$).
\label{ref-5-4-2}
\item We will show in Lemma \ref{ref-2.2-9} below that (\ref{ref-5-4}) implies
\begin{equation}
\label{ref-2.1-6}
\pi_{s\ast}\uRHom_{X/G}\left(P_{G,\Lscr},\RInd^{G}_{G^{\lambda,+}_e}
(V_{G_e^\lambda}(\chi) \otimes_k j_\ast\Oscr_{X^{\lambda,+}})\right)=0
\end{equation}
where $j$ is the inclusion $X^{\lambda,+}\hookrightarrow X$.
Consider the map
\begin{equation}
\label{ref-2.2-7}
P_{G,\chi}=
\RInd^G_{G^{\lambda,+}_e}\left(V_{G_e^\lambda}(\chi)
\otimes\Oscr_X\right)\r \RInd^{G}_{G^{\lambda,+}_e} \left(V_{G_e^\lambda}(\chi) \otimes_k j_\ast\Oscr_{X^{\lambda,+}}\right)
\end{equation}
obtained by applying $\RInd_{G^\lambda_e}^G(V_{G_e^\lambda}(\chi) \otimes-)$ to the obvious map
$
\Oscr_X\r j_\ast \Oscr_{X^{\lambda,+}}
$. Combining \eqref{ref-2.2-7}  with \eqref{ref-2.1-6}
we obtain  from the axioms of triangulated categories a canonical map
\begin{multline}
\label{ref-2.3-8}
\cone
\left(
\pi_{s\ast}\uHom_{X/G}(P_{G,\Lscr},P_{G,\chi})\Lotimes_{\pi_{s\ast}\uEnd_{X/G}(P_{G,\Lscr})} P_{G,\Lscr} \r P_{G,\chi}
\right)\\
\r
\RInd^{G}_{G^{\lambda,+}_e}
\left(V_{G_e^\lambda}(\chi) \otimes_k j_\ast\Oscr_{X^{\lambda,+}}\right).
\end{multline}
We require that \eqref{ref-2.3-8} is an isomorphism.
\end{enumerate}
\end{definitions}
The following lemma is necessary to complete Definition \ref{ref-2.1-2}.
\begin{lemmas} \label{ref-2.2-9}
Assume that $\Lscr$ is as in Definition \ref{ref-2.1-2}(\ref{ref-5-4}). Then
\eqref{ref-2.1-6} holds.
\end{lemmas}
\begin{proof}
  By using an affine covering of $X\quot G$ we may assume that $X$ is
  affine.  By adjointness we have
\begin{multline}
\label{ref-2.4-10}
\RHom_{X/G}(P_{G,\Lscr},\RInd^{G}_{G^{\lambda,+}_e}
(V_{G_e^\lambda}(\chi) \otimes_k j_\ast\Oscr_{X^{\lambda,+}}))\\=
\bigoplus_{\mu\in \Lscr}\Hom_{G_e}(\Res^G_{G_e} \Ind^G_{G_e} V_{G_e}(\mu),
\RInd^{G_e}_{G^{\lambda,+}_e}(V_{G^\lambda_e}(\chi)\otimes_k k[X^{\lambda,+}])).
\end{multline}
By  the $G/G_e$-invariance of $\Lscr$, the simple summands of $(\Res^G_{G_e} \Ind^G_{G_e} V_{G_e}(\mu))_{\mu\in\Lscr}$
are precisely the $(V_{G_e}(\mu))_{\mu\in\Lscr}$ (see (\ref{eq:restriction}, \ref{eq:restriction1})). In other words it suffices to prove that for every $\mu$ such that
$\langle \lambda,\mu\rangle>\langle \lambda,\chi\rangle$
one has
\[
\Hom_{G_e}( V_{G_e}(\mu),
\RInd^{G_e}_{G^{\lambda,+}_e}(V_{G^\lambda_e}(\chi)\otimes_k k[X^{\lambda,+}]))=0
\]
Note
\begin{align}\label{BQ}
\RInd^{G_e}_{G^{\lambda,+}_e}(V_{G^\lambda_e}(\chi)\otimes_k k[X^{\lambda,+}]))
&=\RInd^{G_e}_{G^{\lambda,+}_e}(\RInd_{B}^{G_e^{\lambda,+}}(\chi\otimes_k k[X^{\lambda,+}]))\\\nonumber
&=\RInd^{G_e}_{G^{\lambda,+}_e}\RInd_{B}^{G_e^{\lambda,+}}(\chi\otimes_k k[X^{\lambda,+}])\\\nonumber
&=\RInd^{G_e}_{B}(\chi\otimes_k k[X^{\lambda,+}]).
\end{align}
Using the fact that the weights $\mu$ of  $k[X^{\lambda,+}]$ all satisfy $\langle \lambda,\mu\rangle \le 0$ (see
\S\ref{ref-1.2-0}) we conclude as in the proof of \cite[Lemma 11.2.1]{SVdB} that the cohomology of $\RInd^{G_e}_{B}(\chi\otimes_k k[X^{\lambda,+}])$ are direct sums of $V_{G_e}(\mu)$ with $\langle \lambda,\mu\rangle \le \langle \lambda,\chi\rangle$.
This finishes the proof.
\end{proof}

\subsection{Reduction settings and { $\boldmath \uRHom$}}
\label{sec:mainred}
The following technical result will be our main application of reduction settings.
\begin{propositions}  \label{Di}
Assume that we have a reduction setting $(G,B,T,X,\Lscr,\chi,\lambda)$
and assume $\chi'\in X(T)^+$ is such that
 $\la\lambda,\chi\ra=\la\lambda,\chi'\ra$ and
  $\la\lambda,\bar{g}(\chi')\ra>\la\lambda,\chi\ra$ for all $\bar{g}\not\in (G/G_e)^\lambda$.
Let $i_\lambda:X^\lambda\quot\open^\lambda \r X\quot G$ (see \S\ref{sec:definition} for notation)
be induced from $X^\lambda\r X$. It is clear
that $i_\lambda$ is affine so $i_{\lambda\ast}$ is exact.
Let $\pi_{s,\lambda}$ be the canonical map $X^\lambda/\open^\lambda\r
X^\lambda\quot \open^\lambda$.

We have  isomorphisms
\begin{multline}
\label{eq:rhom}
\pi_{s\ast}\uRHom_{X/G}(\RInd^{G}_{G^{\lambda,+}_e} (V_{G_e^\lambda}(\chi) \otimes_k \Oscr_{X^{\lambda,+}}),
\RInd^{G}_{G^{\lambda,+}_e} (V_{G_e^\lambda}(\chi') \otimes_k \Oscr_{X^{\lambda,+}}))\\
\cong
i_{\lambda\ast}\pi_{s,\lambda\ast}\uRHom_{X^\lambda/\open^\lambda}
(P_{H^\lambda,X^\lambda,\chi},P_{H^\lambda,X^\lambda,\chi'})
\end{multline}
Moreover such isomorphisms are compatible with composition when applicable.
\end{propositions}
\begin{proof} The right-hand side of \eqref{eq:rhom} only has cohomology in degree zero. Hence it is sufficient to construct an isomorphism like \eqref{eq:rhom} in the affine case, in a way which is compatible with
restriction. So we now assume $X$ is affine. Using \eqref{ref-2.1-6} (applied with $\chi$ replaced by $\chi'$) and \eqref{ref-2.3-8} we may replace the first argument
to $\uRHom_{X/G}(-,-)$ in \eqref{eq:rhom} by $P_{G,\chi}$ and as $P_{G,\chi}\cong \Ind_{G_e}^GV_{G_e}(\chi)\otimes_k k[X]$, by adjointness we have to
construct an isomorphism
\begin{multline}\label{firstdisplay}
\RHom_G(\Ind^G_{G_e}V_{G_e}(\chi),
\RInd^{G}_{G^{\lambda,+}_e} (V_{G_e^\lambda}(\chi') \otimes_k k[X^{\lambda,+}]))\\
\cong \Hom_{\open^\lambda}(\Ind_{G_e^{\lambda}}^{\open^\lambda}V_{G_e^\lambda}(\chi),\Ind_{G_e^{\lambda}}^{\open^\lambda}V_{G_e^\lambda}(\chi')\otimes_k k[X^\lambda]).
\end{multline}
We do this next. We have
\begin{multline}\label{7.1.e1}
\RHom_G(\Ind^G_{G_e}V_{G_e}(\chi),
\RInd^{G}_{G^{\lambda,+}_e} (V_{G_e^\lambda}(\chi') \otimes_k k[X^{\lambda,+}]))\\
\cong
\Hom_{G_e^{\lambda,+}}(\Res_{G_e^{\lambda,+}}^{G_e} \Res^G_{G_e}\Ind_{G_e}^{G}V_{G_e}(\chi),
V_{G_e^\lambda}(\chi')\otimes_k k[X^{\lambda,+}])
\\
\overset{\text{\eqref{eq:restriction}}}{\cong}
\Hom_{G_e^{\lambda,+}}(\Res_{G_e^{\lambda,+}}^{G_e} \bigoplus_{\bar{g}\in G/G_e}
{}_{\sigma_{\bar{g}}} V_{G_e}(\chi),V_{G_e^\lambda}(\chi')\otimes_k k[X^{\lambda,+}])
\\\cong
\bigoplus_{\bar g\in \faktor{\tilde{G}^\lambda}{G_e}}\Hom_{G_e^{\lambda,+}}(V_{G_e}(\chi\circ \sigma_{\bar{g}}),V_{G_e^\lambda}(\chi')\otimes_k k[X^{\lambda,+}]),
\end{multline}
where the last isomorphism follows by \eqref{eq:restriction1}, the assumption
$\la\lambda,\bar{g}(\chi)\ra>\la\lambda,\chi'\ra$ for all $\bar{g}\not\in (G/G_e)^\lambda=
\tilde{G}^{\lambda}/G_e$ and the fact that all the weights $\mu$ of $k[X^{\lambda,+}]$ satisfy $\langle \lambda,\mu\rangle\le 0$ by \S\ref{ref-1.2-0}.

Assume $\bar{g}\in (G/G_e)^\lambda=\tilde{G}^{\lambda}/G_e$. By Lemma \ref{lem:recall}, as a $G_e^{\lambda}$-representation
$V_{G_e}(\chi\circ \sigma_{\bar{g}})$ is a direct sum of
$V_{G_e^\lambda}(\chi\circ \sigma_{\bar{g}})$ and representations of the form
$V_{G_e^{\lambda }}(\mu)$ with $\langle \lambda,\mu\rangle > \langle
\lambda,\chi\circ \sigma_{\bar{g}}\rangle=\langle
\lambda,\chi\rangle$. For similar reasons as above such $V_{G^\lambda_e}(\mu)$ cannot contribute to the right-hand side of \eqref{7.1.e1} so that \eqref{7.1.e1} is isomorphic to
\begin{multline}
\label{eq:step2}
\bigoplus_{\bar g\in \faktor{\tilde{G}^\lambda}{G_e}}\Hom_{G_e^{\lambda,+}}(V_{G^\lambda_e}(\chi\circ \sigma_{\bar{g}}),V_{G_e^\lambda}(\chi')\otimes_k k[X^{\lambda,+}])\\
\cong
\bigoplus_{\bar g\in \faktor{\tilde{G}^\lambda}{G_e}}\Hom_{G_e^{\lambda}}(V_{G^\lambda_e}(\chi\circ \sigma_{\bar{g}}),V_{G_e^\lambda}(\chi')\otimes_k k[X^{\lambda}])
\end{multline}
where the second isomorphism follows again by considering $\lambda$-weights and \eqref{eq:part0}. To finish the proof we recall that by definition $\tilde{G}^\lambda/G_e=\open^\lambda/G^\lambda_e$ and
by \eqref{eq:restriction}
\[
\bigoplus_{\bar g\in \open^\lambda/G^\lambda_e} V_{G^\lambda_e}(\chi\circ \sigma_{\bar{g}})
=\Res^{\open^\lambda}_{G^\lambda_e}\Ind^{\open^\lambda}_{G^\lambda_e} V_{G_e^\lambda}(\chi).
\]
It now suffice to apply the adjunction $(\Res^{\open^\lambda}_{G^\lambda_e},\Ind^{\open^\lambda}_{G^\lambda_e})$
to the right-hand side of \eqref{eq:step2} to obtain \eqref{firstdisplay}.

The compatibility with composition is a straightforward but tedious verification.
\end{proof}
\subsection{Reduction to unit components}
We have the following convenient fact.
\begin{propositions}\label{ref-prop-3.4}
Assume that $G$ is a reductive group acting on a smooth variety~$X$.
 If
$(G_e,B,T,X,\Lscr,\chi,\lambda)$ is a reduction setting then so is $(G,B,T,X,\Lscr,\chi,\lambda)$.
\end{propositions}
\begin{proof}
We only have to verify that \eqref{ref-2.3-8} is an isomorphism and to do so we may assume that $X$ is affine.
We use Lemma \ref{lem:criterion} below.
Assume that \eqref{eq:2.3-8} holds for $G=G_e$. Then applying the  functor $\Ind^G_{G_e}$
yields that it holds for $G$.
\end{proof}
\begin{lemmas}
\label{lem:criterion}
Assume that $X$ is affine and assume that conditions (\ref{ref-1-3}-\ref{ref-5-4-2}) in Definition \ref{ref-2.1-2} hold. Let $\langle \Pscr_{G,\Lscr}\rangle$ be the smallest cocomplete subcategory of $D_{\Qch}(X/G)$
containing $P_{G,\Lscr}$.
Then
\eqref{ref-2.3-8} is an isomorphism if and only if
\begin{equation}
\label{eq:2.3-8}
P^\bullet:=\cone(P_{G,\chi}\r \RInd^{G}_{G^{\lambda,+}_e}(V_{G_e^\lambda}(\chi) \otimes_k j_\ast\Oscr_{X^{\lambda,+}}))[-1]\in \langle \Pscr_{G,\Lscr}\rangle .
\end{equation}
Moreover in that case $P^\bullet\cong \Hom_{X/G}(P_{G,\Lscr},P_{G,\chi})\Lotimes_{\End_{X/G}(P_{G,\Lscr})} P_{G,\Lscr}$. In particular\footnote{Replacing $\Hom_{X/G}(P_{G,\Lscr},P_{G,\chi})$ with its projective
resolution over $\End_{X/G}(P_{G,\Lscr})$.}~$P^\bullet$ may be represented by a complex in degrees $\le 0$ whose entries are direct sums of $P_{\mu}$, $\mu\in \Lscr$.
\end{lemmas}
\begin{proof}
Since $P_{G,\Lscr}$ is compact by Theorem \ref{th:st}\eqref{c2}
the inclusion functor $\langle \Pscr_{G,\Lscr}\rangle\hookrightarrow D_{\Qch}(X/G)$ has a right adjoint by Brown representability \cite{Neeman}. Moreover one checks that it is explicitly given by
\[
R=\RHom_{X/G}(P_{G,\Lscr},-)\Lotimes_{\End_{X/G}(P_{G,\Lscr})} P_{G,\Lscr}.
\]
So \eqref{ref-2.3-8} is an isomorphism
if and only if we have  a distinguished triangle\footnote{Here we use that $X$ is affine to identify  global and local $\Hom$'s.}
\[
R(P_{G,\chi})\r P_{G,\chi}\r \RInd^{G}_{G^{\lambda,+}_e}(V_{G_e^\lambda}(\chi) \otimes_k j_\ast\Oscr_{X^{\lambda,+}})\r
\]
where the first arrow is the counit.
This in turn is true if and only if
\begin{equation}
\label{eq:iso1}
R(P_{G,\chi})\r \cone(P_{G,\chi}\r \RInd^{G}_{G^{\lambda,+}_e}(V_{G_e^\lambda}(\chi) \otimes_k j_\ast\Oscr_{X^{\lambda,+}}))[-1]\r
\end{equation}
is an isomorphism (and in that case $P^\bullet\cong R(P_{G,\chi})$). Clearly if \eqref{eq:iso1} is an isomorphism then \eqref{eq:2.3-8} holds.
Conversely assume that \eqref{eq:2.3-8} holds.
 Then \eqref{eq:iso1} is a morphism in $\langle P_{G,\Lscr}\rangle$. To test if it is an isomorphism
we may apply $\RHom_{X/G}(P_{G,\Lscr},-)$. But then \eqref{eq:iso1} becomes the identity by \eqref{ref-2.1-6}.
Hence we are done.
\end{proof}
\subsection{Reduction to closed subschemes}\label{subschemes}
We will create reduction settings first in the linear case ($X$ being a representation) and then
we will restrict them to closed subschemes.
To do this will use the following theorem:
\begin{theorems}
\label{ref-3.1-13}
Let $G$ be a  reductive group and let $Y\subset X$ be a  closed embedding
of smooth  $G$-varieties.  If
$(G,B,T,X,\Lscr,\chi,\lambda)$ is a reduction setting then so is
$(G,B,T,Y,\Lscr,\chi,\lambda)$.
\end{theorems}
To prove this we may assume that $X$ is affine. We first discuss a special case. 
\begin{lemmas}
\label{ref-3.2-14} Assume that $(G,X,Y)$ are as in the statement of Theorem \ref{ref-3.1-13} but
with $X$ affine. Assume in addition
that there is a $G_m$-action on $X$ in a way which commutes with the $G$-action such that
$k[X]$ has only weights $\ge 0$  and\footnote{If the $G_m$-action is denoted by $\gamma$ then this condition may also be written as
$Y=X^\gamma$, $X=X^{\gamma,-}$.} $k[Y]=k[X]^{G_m}$. Then the conclusion of Theorem \ref{ref-3.1-13} holds.
\end{lemmas}
\begin{proof} Assume that $(G,B,T,X,\Lscr,\chi,\lambda)$ is a reduction setting. If we replace the $\End_{X/G}(P_{G,\Lscr})$-module $\Hom_{X/G}(P_{G,\Lscr},P_{G,\chi})$ in \eqref{ref-2.3-8} by its projective resolution we  obtain  a resolution
\begin{equation}
\label{ref-3.1-15a}
P^\bullet_{X}\cong \cone(P_{G,X,\chi}\r \RInd^{G}_{G^{\lambda,+}_e}
(V_{G_e^\lambda}(\chi) \otimes_k k[X^{\lambda,+}]))[-1]
\end{equation}
as in Lemma \ref{lem:criterion} (we have switched to coordinate ring notation). Moreover we may assume
that this resolution is $G_m$-equivariant. In addition since $k[X^{\lambda,+}]$ is a quotient of $k[X]$ it only
has $G_m$-weights $\ge 0$ and this property is not affected by applying $\RInd^{G}_{G^{\lambda,+}_e}
(V_{G_e^\lambda}(\chi) \otimes_k -)$. We conclude that as $G\times G_m$-equivariant $k[X]$-module $P^n_X$ may be
assumed to be a direct sum of $P_{G,X,\mu}\otimes_k \sigma_n$, $n\ge 0$  where $\sigma_n$ is the $G_m$-character $z\mapsto z^n$ and $\mu\in \Lscr$.
We then have
\[
(P_{G,X,\mu}\otimes_k \sigma_n)^{G_m}=
\begin{cases}
P_{G,Y,\mu}&\text{if $n=0$}\\
0&\text{if $n>0$}
\end{cases}
\]
Taking $G_m$-invariants of \eqref{ref-3.1-15a} we now get a similar resolution
\[
P^\bullet_{Y}\cong\cone( P^\bullet_{G,Y,\chi}\r \RInd^{G}_{G^{\lambda,+}_e}
(V_{G_e^\lambda}(\chi) \otimes_k k[X^{\lambda,+}])^{G_m}[-1].
\]
Furthermore since the $G$ and $G_m$-action do not interfere with each other we have
\[
 \RInd^{G}_{G^{\lambda,+}_e} (V_{G_e^\lambda}(\chi) \otimes_k k[X^{\lambda,+}])^{G_m}= \RInd^{G}_{G^{\lambda,+}_e} (V_{G_e^\lambda}(\chi) \otimes_k k[X^{\lambda,+}]^{G_m})
\]
and finally using the description of $k[X^{\lambda,+}]$ in \S\ref{ref-1.2-0} we easily see that
$k[X^{\lambda,+}]^{G_m}=k[Y^{\lambda,+}]$. We conclude that \eqref{eq:2.3-8} holds for $Y$. This finishes the proof by Lemma \ref{lem:criterion}.
\end{proof}
We will now reduce the proof of Theorem \ref{ref-3.1-13} to the special case considered in Lemma \ref{ref-3.2-14} using the Luna slice
theorem.
\begin{lemmas}
\label{ref-3.3-16} Let $G$ be a linear algebraic group acting on an affine variety
  $X$ and let $P\in D_{\Qch}(X/G)$. Assume that $P$ is zero in the neighborhood
of any point with closed orbit, i.e.\ any $x\in X$ such that $Gx$ is closed
has an open neighborhood $U_x$ such that $P{\mid} U_x=0$. Then $P=0$.
\end{lemmas}
\begin{proof}
  Put $U'=\bigcup_x U_x$ and $U=GU'$. Then $P{\mid} U=0$ so it is
  sufficient to prove $U=X$. Assume this is not the case. Since $X-U$
  is closed and $G$-invariant it contains a closed orbit (e.g.\ an
  orbit of minimal dimension). This is an obvious contradiction.
\end{proof}

\begin{lemmas} \label{ref-3.4-17} If $(G,B,T,X,\Lscr,\chi,\lambda)$ is
  such that the conditions (\ref{ref-1-3}-\ref{ref-6-5}) from
  Definition \ref{ref-2.1-2} hold, and such that $X$ is affine, then
  write $C_X$ for the cone of \eqref{ref-2.3-8}. If $\alpha:Z\r X$ is
  a strongly \'etale $G$-equivariant morphism
then $\alpha^\ast (C_X)=C_Z$.
\end{lemmas}
\begin{proof} This ultimately boils down to $\alpha^\ast \Oscr_{X^{\lambda,+}}=\Oscr_{Z^{\lambda,+}}$ which is true thanks to Lemma \ref{ref-1.2.1-1}.
\end{proof}
\begin{proof}[Proof of Theorem \ref{ref-3.1-13}]
  We assume that $(G,B,T,X,\Lscr,\chi,\lambda)$ is a reduction setting
and $X$ is affine.  Thus $C_X=0$
  and we have to deduce from it $C_Y=0$.  According to Lemma
  \ref{ref-3.3-16} it suffices to do this in the neighborhood of any
  closed orbit.  So let $Gy$ be a closed orbit in $Y$ and let $N_Y$ be a
  $G_y$-invariant complement to $T_y(Gy)$ in $T_y(Y)$.  Then according to the
  Luna slice theorem \cite{Luna} there is an affine $G_y$-invariant
  ``slice'' $y\in S\subset Y$ to the orbit of $y$ and strongly
  $G_y$-equivariant \'etale morphism $S\r N_Y$ which sends $y$ to $0$
  such that the induced maps
\[
Y\xleftarrow{\alpha} G\times^{G_y} S\xrightarrow{\beta} G\times^{G_y} N_Y
\]
are strongly \'etale. Then by Lemma \ref{ref-3.4-17} we have $\alpha^\ast(C_Y)=\beta^\ast(C_{ G\times^{G_y} N_Y})$.
Thus it is sufficient to prove that $C_{G\times^{G_y} N_Y}=0$.

By assumption $Gy$ is closed in $X$. Let $V$ be a $G_y$-invariant complement to $T_y(Y)$ in $T_y(X)$
and put
$N_X:=N_Y\oplus V$. Since $C_X=0$, by the same reasoning as above we conclude
that $C_{ G\times^{G_y} N_X}$ is zero in a neighborhood of the zero section of $G\times^{G_y} N_X\r G/G_y$. However
note that since $ C_{ G\times^{G_y} N_X}$ is natural, it is in particular equivariant for the scalar $G_m$-action on $N_X$. So in fact $ C_{ G\times^{G_y} N_X}=0$.

Now let $G_m$ act on $N_X=N_Y\oplus V$ by acting trivially on $N_Y$ and with weight $-1$ on~$V$. Then
the inclusion $ G\times^{G_y} N_Y\hookrightarrow G\times^{G_y} N_X$ falls under the setting
considered in Lemma \ref{ref-3.2-14}. We conclude from this lemma that  $C_{G\times^{G_y} N_Y}=0$, finishing the proof.
\end{proof}

\subsection{Reduction settings in the connected linear case}
We use the notation and conventions introduced in \S\ref{linear}.
We need the twisted
Weyl group action of $\Wscr$ on $X(T)$: $w{\ast}\chi:=w(\chi+\bar{\rho})-\bar{\rho}$. If $\chi\in X(T)$
and there is some $w{\ast} \chi$ which is dominant then we write $\chi^+=w{\ast} \chi$. Otherwise
$\chi^+$ is undefined.
\begin{propositions}\label{-iC}
  Let $G$ be a connected reductive group and assume
  $B,T,\chi,\lambda,\Lscr$ satisfy (\ref{ref-1-3}-\ref{ref-5-4}) in Definition
  \ref{ref-2.1-2}.  Let $X=W^\vee$ where $W$ is a $G$-representation
  with weights $\beta_1,\ldots,\beta_d$.  Assume
  $(\chi+\beta_{i_1}+\dots+\beta_{i_{-p}})^+\in \Lscr$ for all
  $\emptyset\neq\{i_1,\ldots,i_{-p}\}\subseteq\{1,\ldots,d\}$,
  $i_j\neq i_{j'}$ for $j\neq j'$ such that $(\chi+\beta_{i_1}+\dots+\beta_{i_{-p}})^+$ is defined.  Then
  $(G,B,T,X,\Lscr,\chi,\lambda)$ is a reduction setting.
\end{propositions}
\begin{proof}
We only have to verify \eqref{ref-2.3-8}.
 We  denote by
$K_\lambda$ the subspace of $W$ spanned by the weight vectors $w_j$
such that $\la \lambda,\beta_j\ra>0$. Note that $\Spec
\Sym(W/K_\lambda)\cong X^{\lambda,+}$.
In \cite[(11.3)]{SVdB} we constructed a quasi-isomorphism
\begin{equation}
\label{quasiiso}
C_{\lambda,\chi}\r \RInd^{G}_{G^{\lambda,+}} (V_{G^\lambda}(\chi) \otimes_k k[X^{\lambda,+}]),
\end{equation}
where $C_{\lambda,\chi}$ is a complex of the form
\begin{align}\label{ref-11.3-98}
C_{\lambda,\chi}\overset{\text{def}}{=}&\left(\bigoplus_{p\le 0,q\ge 0} R^q\Ind^G_B(\chi\otimes_k \wedge^{-p} K_{\lambda})\otimes_{k} k[X][-p-q],d\right).
\end{align}
We
showed that after forgetting the differential $C_{\lambda,\chi}$ is a sum of
$G$-equivariant projective modules of the form ${P}_{\mu}$ where the $\mu$ are among
the weights
\begin{equation}
\label{eq:weights}
(\chi+\beta_{i_1}+\beta_{i_2}+\cdots+\beta_{i_{-p}})^+
\end{equation}
(with each such expression occurring at most once)
where $\{i_1,\ldots,i_{-p}\}\subset\{1,\ldots,d\}$, $i_j\neq i_{j'}$ for
$j\neq j'$ and $\langle\lambda,\beta_{i_j}\rangle>0$ \cite[Lemma 11.2.1]{SVdB}.
Moreover there is a single copy of $P_\chi$ which lives in degree
zero.  It is not explicitly stated in loc.\ cit.\ but it follows easily
from the construction that this copy of $P_\chi$ yields an inclusion $P_\chi\r
C_{\lambda,\chi}$ such that the composition $P_\chi\r
C_{\lambda,\chi}\r \RInd^{G}_{G^{\lambda,+}} (V_{G^\lambda}(\chi)
\otimes_k k[X^{\lambda,+}])$ is the canonical morphism exhibited in \eqref{ref-2.2-7}.

Let $C'_{\lambda,\chi}=C_{\lambda,\chi}/P_\chi$. Then by the fact that \eqref{quasiiso} is a quasi-isomorphism we have $C'_{\lambda,\chi}\cong \cone(P_\chi\to \RInd^{G}_{G^{\lambda,+}} (V_{G^\lambda}(\chi)
\otimes_k k[X^{\lambda,+}]))$.
Since by hypothesis the summands $P_\mu$ of $C'_{\lambda,\chi}$ are summands
of $P_\Lscr$ we find that \eqref{eq:2.3-8} holds and hence
we are done by Lemma \ref{lem:criterion}.
\end{proof}

\section{Partitioning $X(T)^+$}\label{partition}
\subsection{Preliminaries}
\label{sec:prelim}
We assume we are in the setting of \S\ref{linear}.  In particular $G$ is
connected and acts on a representation $X=W^\vee$.
We now introduce some extra notation.
We let $\Phi\subset X(T)$ be the roots of $G$. We
write $\Phi^-$ for the negative roots of $G$ (the roots of $B$) and
$\Phi^+$ for the positive roots. We choose a positive definite $\Wscr$-invariant quadratic form $(-,-)$ on $X(T)_\RR$.
If $\alpha\in \Phi$ then $\check{\alpha}\in Y(T)_\RR$ is the corresponding coroot defined by
$\langle \check{\alpha},\chi\rangle=2(\alpha,\chi)/(\alpha,\alpha)$
and the associated reflection on
$X(T)_\RR$ is defined by $s_\alpha(\chi)=\chi-\langle\check{\alpha},\chi\rangle\alpha$. We put
$\Phi^\vee=\{\check{\alpha}\mid \alpha\in \Phi\}$.

We set $ \Phi_\lambda=\{\alpha\in \Phi\mid
\la\lambda,\alpha\ra=0\}.  $ We denote by
$\Phi_\lambda^+=\Phi^+\cap\Phi_\lambda$ the set of positive roots of
$G^\lambda$.

Let us recall a slightly extended version of \cite[Corollary D.3.]{SVdB} with the same proof.
\begin{lemmas}\label{ref-A.2}
Let  $\lambda\in Y(T)_\RR^-$, $w\in \Wscr$, $\chi\in X(T)$  such that $w{\ast}\chi$ is dominant.  Then
$\la \lambda, w{\ast}\chi\ra\le \la \lambda, \chi\ra$
with equality if and only if $w \in \Wscr_{G^\lambda}$.
\end{lemmas}
\begin{proof}
 When comparing with \cite[Corollary
  D.3.]{SVdB} note that in loc.\ cit. the role of $\lambda$ is played
  by $y$, which is assumed to be dominant, rather than anti-dominant
  which is the case here. So the inequalities are reversed.

  Since $\chi^+:=w{\ast}\chi$ is dominant we have $\chi^+=s_{\alpha_n}{\ast}\cdots
  {\ast}s_{\alpha_1} {\ast}\chi$
such that for each $\chi_i:=s_{\alpha_i}{\ast}\cdots {\ast} s_{\alpha_1}{\ast}\chi$
the inequality $\langle\check{\alpha}_{i+1},\chi_i\rangle\le -2$ holds.

In loc.\ cit.\ it is shown that $\langle\lambda,w{\ast}\chi\rangle\le \langle\lambda ,\chi\rangle$.
Going through the proof we see that the only possibility for equality to occur is when $\langle \lambda,\alpha_i\rangle=0$ for all $i$. But then $w':=s_{\alpha_n}\cdots
  s_{\alpha_1}\in \Wscr_{G^\lambda}$. Since $\chi^+$ is dominant it has trivial stabilizer
for the $\ast$-action. Since $(ww'^{-1})\ast\chi^+=\chi^+$ we conclude $w=w'\in \Wscr_{G^\lambda}$.
\end{proof}
We will also need the following variant.
\begin{lemmas}\label{ref-A.21}
Let  $\lambda\in Y(T)_\RR^-$, $w\in \Wscr$, $\chi\in X(T)$  such that $w\chi$ is dominant.  Then
$\la \lambda, w\chi\ra\le \la \lambda, \chi\ra$
with equality if and only if $w\chi \in \Wscr_{G^\lambda}\chi$.
\end{lemmas}
\begin{proof} The proofs is along the same lines as the proof of Lemma \ref{ref-A.2} except that $w\chi$ may have non-trivial stabilizer. This accounts for the slightly weaker conclusion.
\end{proof}
We define
\begin{align*}
T_\lambda^+=\{i\mid \la\lambda,\beta_i\ra> 0\},\quad
T_\lambda^0=\{i\mid \la\lambda,\beta_i\ra= 0\},\quad
T_\lambda^-=\{i\mid \la\lambda,\beta_i\ra< 0\}.
\end{align*}
A point $x\in X$ is {\em stable} if it has closed orbit and finite
stabilizer.  $X$ has a $T$-stable point if and only if for every
$\lambda\in Y(T)\setminus\{0\}$ there exists $i$ such that
$\la\lambda,\beta_i\ra>0$ (i.e., not all the weights lie in a
half space defined by the hyperplane through the origin).

\medskip

\emph{In the rest of this section we assume that $X=W^\vee$ has a $T$-stable point.}

\subsection{Expression of $\chi$ in terms of faces of $\Sigma$}
\label{sec:interms}
As $X$ has a $T$-stable point, $0$ lies in the interior of the positive span of
$(\beta_i)_i$ and in particular $-\bar \rho+\bigcup_r
r\Sigma=X(T)_\RR$, thus every $\chi\neq -\bar{\rho}\in X(T)$ lies  in the
relative interior of a unique proper face of $-\bar\rho+r\bar{\Sigma}$ for
a unique $r>0$. We will partition the set $X(T)^+$ according to the
relative interiors of faces of $-\bar{\rho}+r\bar{\Sigma}$ to which its elements belong. However for convenience
we will not use the faces directly but rather some equivalent combinatorial data
associated to them.

\medskip

For a set $S$ let $\Pscr(S)$ be its power set.
We put a partial ordering $\prec$
on $\RR^+\times \cP(\{1,\dots,d\})^3$  by declaring $(r,S^+,S^-,S^0)\preceq (r',S^{\prime +},S^{\prime -},S^{\prime 0})$
if either  $r<r'$ or else $r=r'$, $S^+\subset S^{\prime +}$ and $S^-\subset S^{\prime -}$.
If ${\bf S}=(S^+,S^-,S^0)$ then we write $|{\bf S}|=(|S^+|,|S^-|,|S^0|)$.

Let $\RR^+\times \NN^3$ be equipped with the (total) lexicographic ordering.  There
is an order preserving map
\begin{equation}
\label{eq:order}
(\RR^+\times \cP(\{1,\dots,d\})^3,\prec)\r (\RR^+\times \NN^3,<):(r,{\bf S})\mapsto (r,|{\bf S}|)
\end{equation}
whose fibers are incomparable among each other.
\begin{lemmadefinitions}\label{ref-1.5}
\begin{enumerate}
\item
For $\chi\neq -\bar{\rho}\in X(T)$ there exists an expression of the form
\begin{equation}\label{expresschi}
\chi=-\bar\rho-r\sum_{i\in S^+}\beta_i+0\sum_{i\in S^-}\beta_i+\sum_{i\in S^0}b_i\beta_i,
\end{equation}
where  $S^+\neq \emptyset$, $r>0$,
  $\forall i:-r<b_i< 0$ and $S^+\coprod S^-\coprod S^0=\{1,\dots,d\}$.
There is a unique tuple $(r_\chi,{\bf S}_\chi):=(r_{\chi},S_{\chi}^+,S_{\chi}^-,S_{\chi}^0),$
for which $(r_\chi,|{\bf S}_\chi|)$ is minimal among tuples attached to such expression of the form \eqref{expresschi}.
  If $\chi=-\bar{\rho}$
then we put by convention $r_\chi=0$, $S_\chi^+=S_\chi^-=\emptyset$ (although  \eqref{expresschi} is then not true). Below we refer to this situation as the ``trivial case''.
\item\label{tri}
If $\chi\in X(T)^+$ then there exists $\lambda\in Y(T)^-$ with the properties:
$S_{\chi}^{-}=T_\lambda^-$, $S_{\chi}^+=T_\lambda^+$, $S_{\chi}^0=T_\lambda^0$.
\end{enumerate}
\end{lemmadefinitions}
\begin{remarks} If $\chi\in X(T)^+$ then the trivial case is equivalent to $G=T$ and $\chi=0$.
\end{remarks}
\begin{remarks}\label{rmk:faces} It will follow from the proof of Lemma \ref{ref-1.5} as well a
 Lemma \ref{eq:supporting} below that
the data $(r_\chi,{\bf S}_\chi)$ identifies which
  proper face $F$ of $-\bar{\rho}+r_\chi\Sigma$ contains~$\chi$ in its relative
  interior. Reformulating Lemma \ref{eq:supporting} one has
\begin{equation}
\label{eq:relint}
\relint F=\biggl\{ -\bar\rho-r_{\chi}\sum_{i\in S_{\chi}^+}\beta_i+\sum_{i\in S_\chi^0}b_i\beta_i\mid \forall i:-r_\chi<b_i<0\biggr\}
\end{equation}
Furthermore $\lambda$ as in Lemma \ref{ref-1.5}\eqref{tri} defines an appropriately chosen
   supporting plane for that face. Finally the $\prec$-ordering is opposite to the ordering given by inclusion of faces.
\end{remarks}
\begin{proof}[Proof of Lemma \ref{ref-1.5}]
  The existence of an expression with minimal $(r_\chi,|{\bf
    S}_\chi|)$ is obvious.  To prove the uniqueness of the associated
  tuple $(r_{\chi},{\bf S}_{\chi})$ assume that there are two minimal
  expressions with different associated tuples.  Taking their average
  we obtain an expression which is strictly smaller than both the original
  expressions, contradicting the minimality.

We will now prove \eqref{tri}. If we are in the trivial
case
then we take
$\lambda=0$.
So we will now assume we are not in the trivial case.
We take $r$ minimal such that $\chi\in -\bar \rho+r\bar\Sigma$ (and hence $r_{\chi}=r$).
By \cite[Lemma C.2]{SVdB} (and as all $\beta_i\in X(T)$) there exists $0\neq \lambda\in Y(T)$ 
such that $\la\lambda,\chi\ra<\la\lambda,\mu\ra$ for all $\mu\in -\bar\rho+r_{\chi}\Sigma$
 and $\chi$ can be written as
\begin{equation}
\label{eq:chiexpr}
\chi=-\bar\rho-r_{\chi}\sum_{i\in T_{\lambda}^+}\beta_i+\sum_{i\in T_{\lambda}^0}b_i\beta_i,
\end{equation}
$-r_{\chi}<b_i<0$. Thus by Lemma \ref{eq:supporting} below $\chi+\bar{\rho}$ is
in the relative interior of the face of $r_\chi\bar{\Sigma}$ defined
by the supporting half plane
$\la\lambda,\chi+\bar{\rho}\ra\le\la\lambda,-\ra$. Let $w\in \Wscr$ be such
that $w{\lambda}\in Y(T)^-$. By the discussion preceding the \cite[
(11.4)]{SVdB} $\la w\lambda,\chi+\bar{\rho}\ra\le\la w\lambda,-\ra$ is still a
supporting half plane for $r_\chi\bar{\Sigma}$. It is easy to see
that the corresponding face must be equal to $wF$.
Since the face still contains $\chi+\bar{\rho}$ we must have $F=wF$.
It follows again
from Lemma \ref{eq:supporting} below that \eqref{eq:chiexpr} remains
true with ${\lambda}$ replaced by $w{\lambda}$. So we now assume ${\lambda}\in Y(T)^-$.

By Observation (3)
in the proof of \cite[Theorem 1.4.1]{SVdB} (applied to both boundaries of the interval $[-r_\chi,0]$)
we find that in any expression of the form \eqref{expresschi} we must have $T_{\lambda}^+\subset S^+$,
$T_{\lambda}^-\subset S^-$. Since $(|S_\chi^+|,|S_\chi^-|)$ is minimal this implies
$S_{\chi}^\pm=T_{{\lambda}}^\pm$, $S_{\chi}^0=T_{\lambda}^0$, establishing \eqref{tri}.
\end{proof}
We have used the following result.
\begin{lemmas} \label{eq:supporting}
Let $\chi\in X(T)_\RR$, $0\neq \lambda\in Y(T)_\RR$ be such that 
the equation $\langle {\lambda},\chi\rangle = \langle {\lambda},-\rangle$ defines a supporting half-plane of $\bar{\Sigma}$. Then $\chi$ is in the relative interior of the corresponding face if and only if $\chi$ can be written as
\[
\chi=-\sum_{i\in T_{\lambda}^+}\beta_i+\sum_{i\in T_{\lambda}^0}b_i\beta_i,
\]
with $-1<b_i<0$.
\end{lemmas}
\begin{proof} Let $H$ be the hyperplane $\langle {\lambda},\chi\rangle = \langle {\lambda},-\rangle$. Then
$
F=H\cap \bar{\Sigma}
$
is given by those $\mu=\sum_i c_i\beta_i$ such that $\langle {\lambda},\chi\rangle = \langle {\lambda},\mu\rangle$
and
\begin{align*}
i\in T^+_{\lambda} &\Rightarrow c_i=-1,\\
i\in T^-_{\lambda} &\Rightarrow c_i=0,\\
i\in T^0_{\lambda} &\Rightarrow c_i\in [-1,0].
\end{align*}
This follows in fact from Observation (3)
in the proof of \cite[Theorem 1.4.1]{SVdB} (applied to both boundaries of the interval $[-1,0]$). It
follows that $\relint F$ is given by those $\mu=\sum_i c_i\beta_i$ such that $\langle {\lambda},\chi\rangle = \langle {\lambda},\mu\rangle$
and
\begin{align*}
i\in T^+_{\lambda} &\Rightarrow c_i=-1,\\
i\in T^-_{\lambda} &\Rightarrow c_i=0,\\
i\in T^0_{\lambda} &\Rightarrow c_i\in ]-1,0[.
\end{align*}
Applying this with $\mu=\chi$ yields what we want.
\end{proof}
\begin{lemmas} \label{ref-1.5bis}
Assume that $\lambda\in Y(T)^-$ corresponds to $\chi\in X(T)^+$ as in Lemma \ref{ref-1.5}\eqref{tri}. Then the following properties hold
\begin{enumerate}
\item\label{ena}
Let $\mu \in X(T)^+$. If $({r}_{\chi},{\bf S}_\chi)\not\preceq({r}_{\mu},{\bf S}_{\mu})$ then
$\la\lambda,\chi\ra<\la\lambda,\mu\ra$.
Similarly if $({r}_{\chi},{\bf S}_\chi)=({r}_{\mu},{\bf S}_{\mu})$  then
$\la\lambda,\chi\ra=\la \lambda,\mu\ra$.
\item\label{dva}
the sets $\{\beta_i\mid i\in S_{\chi}^+\}$, $\{\beta_i\mid i\in S_{\chi}^-\}$ are $\cW_{G^\lambda}$-invariant.
\end{enumerate}
\end{lemmas}
\begin{proof}
\eqref{dva} is immediate from Lemma \ref{ref-1.5}\eqref{tri}
since $\lambda$ is stabilized by $\Wscr_{G^\lambda}$.

Now we verify \eqref{ena}. Again it is sufficient to consider the
non-trivial case. The second claim is easy so we
discuss the first one.  We have
\[
\mu=-\bar\rho-r_\mu\sum_{i\in S^+_\mu}\beta_i+\sum_{i\in S^0_\mu}b_i\beta_i,
\]
with $-r_\mu<b_i<0$. Write $\ss_i=\langle \lambda,\beta_i\rangle$. Then using Lemma \ref{ref-1.5}\eqref{tri}
\begin{equation}
\label{eq:scalar}
\langle \lambda ,\chi\rangle=-\langle \lambda,\bar{\rho}\rangle
-r_\chi\sum_{i\in T_\lambda^+}\ss_i
\end{equation}
and
\begin{equation}
\label{eq:chain}
\begin{aligned}
\langle \lambda,\mu\rangle&=
-\langle \lambda,\bar{\rho}\rangle-r_\mu\sum_{i\in S^+_\mu}\ss_i+
\sum_{i\in S^0_\mu}b_i\ss_i\\
&=-\langle \lambda,\bar{\rho}\rangle-r_\mu\sum_{i\in S^+_\mu\cap T_\lambda^+}\ss_i
-r_\mu\sum_{i\in S^+_\mu\cap T_\lambda^-}\ss_i+
\sum_{i\in S^0_\mu\cap T^+_\lambda}b_i\ss_i
+\sum_{i\in S^0_\mu\cap T^-_\lambda}b_i\ss_i\\
&\ge -\langle \lambda,\bar{\rho}\rangle-r_\mu\sum_{i\in S^+_\mu\cap T_\lambda^+}\ss_i
-0\sum_{i\in S^+_\mu\cap T_\lambda^-}\ss_i
-r_\mu\sum_{i\in S^0_\mu\cap T^+_\lambda}\ss_i
+0\sum_{i\in S^0_\mu\cap T^-_\lambda}\ss_i\\
&= -\langle \lambda,\bar{\rho}\rangle-r_\mu\sum_{i\in (S^+_\mu\cup S^0_\mu)\cap T_\lambda^+}\ss_i\\
&\ge -\langle \lambda,\bar{\rho}\rangle-r_\mu\sum_{i\in T_\lambda^+}\ss_i.
\end{aligned}
\end{equation}
The total inequality will be strict if any of the following conditions hold:
\begin{align*}
S^+_\mu\cap T_\lambda^-&\neq \emptyset\qquad \text{ or}\\
S^0_\mu\cap (T_\lambda^+\cup T^-_\lambda)&\neq \emptyset\qquad \text{ or}\\
(S^+_\mu\cup S^0_\mu)&\not\supset T_\lambda^+,
\end{align*}
which is equivalent to any of the following conditions holding
\begin{equation}
\label{eq:strict}
\begin{aligned}
S^+_\mu&\not\supset T_\lambda^+\qquad\text{ or}\\
S^-_\mu&\not\supset T_\lambda^-\qquad\text{ or}\\
S^0_\mu&\not\supset T_\lambda^0.
\end{aligned}
\end{equation}
To prove \eqref{ena} we have to show
\[
(r_\chi,{\bold S}_\chi)\not\preceq (r_\mu,{\bold S}_\mu)\Rightarrow \langle \lambda,\chi\rangle<\langle \lambda,\mu\rangle.
\]
The condition on the left hand side is equivalent to any of the following conditions holding
\begin{equation}
\label{eq:partial}
\begin{gathered}
r_\chi>r_\mu\qquad \text{ or}\\
r_\chi=r_\mu\text{ and } T^+_\lambda\not\subset S^+_\mu\qquad \text{ or}\\
r_\chi=r_\mu\text{ and } T^-_\lambda\not\subset S^-_\mu.
\end{gathered}
\end{equation}
If $r_\chi>r_\mu$ then $\langle \lambda,\chi\rangle<\langle \mu,\chi\rangle$ by comparing
\eqref{eq:chain} to \eqref{eq:scalar} using the fact that $S^+_\chi=T^+_\lambda\neq \emptyset$.
If $r_\chi=r_\mu$ then \eqref{ena} follows by comparing \eqref{eq:strict} with
\eqref{eq:partial}.
\end{proof}
\begin{corollarys}\label{betaplus}
Let $\chi\in X(T)^+$ be such that $r_\chi\ge 1$ and let $\lambda\in Y(T)^-$ be as in Lemma \ref{ref-1.5}(\ref{tri}).
If $p>0$ and $\mu=\chi+\beta_{i_1}+\cdots+\beta_{i_{p}}$, where $\{i_1,\ldots,i_{p}\}\subset\{1,\ldots,d\}$, $i_j\neq i_{j'}$ for
$j\neq j'$ and $\la\lambda,\beta_{i_j}\ra>0$,
then $({r}_{{\mu}^+},{\bf |S}_{{\mu}^+}|)<({ r}_\chi,{\bf |S}_\chi|)$.
Moreover, $\la\lambda,\chi\ra<\la\lambda,\mu^+\ra$.
\end{corollarys}

\begin{proof}
By the property \eqref{tri} in Lemma \ref{ref-1.5} every $k$ for which $\la\lambda,\beta_k\ra>0$ belongs to $S_{\chi}^+$ and thus ${r}_{\mu}<{ r}_{\chi}$ or  ${r}_{\mu}= {r}_{\chi}$ and $|{\bf S}_{\mu}|<|{\bf S}_{\chi}|$.
In other words $(r_\mu,|{\bf S}_\mu|)<(r_{\chi},|{\bf S}_\chi|)$.

As (${r}_{\mu}$, $|{\bf S}_{\mu}|$) depends only on the $\cW$-orbit of $\mu$ for the $*$-action,
we also have $(r_{\mu^+},|{\bf S}_{\mu^+}|)<(r_{\chi},|{\bf S}_\chi|)$ and hence $(r_{\mu^+},{\bf S}_{\mu^+})
\not\succeq(r_{\chi},{\bf S}_\chi)$.
Property \eqref{ena} in Lemma \ref{ref-1.5bis} then implies $\la \lambda,\mu^+\ra>\la\lambda,{\chi}\ra$.
\end{proof}
We denote
\begin{align*}
\chi_p&=-r_{\chi}\sum_{i\in S_{\chi}^+}\beta_i
\end{align*}
The following lemma gives a description of the set of $\chi$ with given $(r_\chi,{\bf S}_\chi)$
in terms of objects related to $G^\lambda$.
\begin{lemmas} \label{chi_lambda}
Let $\chi\in X(T)^+$ and let $\lambda$ be as in Lemma \ref{ref-1.5}(\ref{tri}).
Then the set
\begin{equation}
\label{eq:set}
\{\chi'\in X(T)^+\mid (r_{\chi'},{\bf S}_{\chi'})= (r_{\chi},{\bf S}_{\chi})\}
\end{equation}
is equal to
\[
(\nu-\bar{\rho}_\lambda+r_\chi \Sigma^0_\lambda)\cap X(T)^\lambda
\]
where
\[
\nu= -\bar\rho+\bar\rho_\lambda+\chi_p.
\]
Moreover $\nu$ is $\Wscr_{G^\lambda}$-invariant.
\end{lemmas}
\begin{proof}
  The fact that $\nu$ is $\Wscr_{G^\lambda}$-invariant is a direct
  consequence of Lemma \ref{ref-1.5bis} \eqref{dva} and the standard fact that $-\bar{\rho}+\bar{\rho}_\lambda$ is $\Wscr_{G^\lambda}$-invariant.

By Lemma \ref{ref-1.5} and \eqref{eq:relint} we have
\[
\{\chi'\in X(T)^+\mid (r_{\chi'},{\bf S}_{\chi'})= (r_{\chi},{\bf S}_{\chi})\}=(-\bar{\rho}+\chi_p+r_\chi\Sigma^0_\lambda)\cap X(T)^+.
\]
By Lemma \ref{chi_lambdabis} below we may rewrite this as
\[
\{\chi'\in X(T)^\lambda\mid (r_{\chi'},{\bf S}_{\chi'})= (r_{\chi},{\bf S}_{\chi})\}
=
(-\bar{\rho}+\chi_p+r_\chi\Sigma^0_\lambda)\cap X(T)^\lambda
\]
which is the same as
\[
(\nu-\bar{\rho}_\lambda+r_\chi\Sigma^0_\lambda)\cap X(T)^\lambda.\qed
\]
\def\qed{}\end{proof}
We have used the following result.
\begin{lemmas}\label{chi_lambdabis}
Let
$\chi\in X(T)^+$ and let $\lambda$ be as in Lemma \ref{ref-1.5}(\ref{tri}).
If $\chi\in X(T)^+$ and $\chi'\in X(T)^\lambda$ are such that
$({r}_{\chi'},{\bf S}_{\chi'})=({r}_{\chi},{\bf S}_{\chi})$ then $\chi'\in X(T)^+$.
\end{lemmas}
\begin{proof}
We assume that we are not in the trivial case otherwise there is nothing to do.
We first verify $s_\alpha{*}\chi'\neq\chi'$
for all $\alpha\in \Wscr$.  This implies that $\chi^{\prime+}$ exists.
Assume on the contrary that $s_\alpha{*}\chi'=\chi'$ for some
$\alpha$.  By the uniqueness of the minimal expression in Lemma
\ref{ref-1.5} we obtain that
$S_{\chi'}^+=S^+_{\chi}=T_\lambda^+$ is $s_\alpha$-invariant. Using
the formula \eqref{eq:scalar} we find $\langle
\lambda,\chi\rangle=\langle \lambda,s_\alpha{\ast}\chi\rangle$.
Then Lemma \ref{ref-A.2} implies
$s_\alpha\in \Wscr_{G^\lambda}$. However this is  excluded by the
fact that $s_\alpha{\ast}\chi'=\chi'$ and $\chi'\in X(T)^\lambda$.

Assume $\chi'\not \in X(T)^+$. By the above discussion there exists $1\neq w\in \Wscr$ such that
$w{\ast}\chi^{\prime}$ is dominant.
Furthermore  $w\not\in \cW_{G^\lambda}$ as $\chi'\in X(T)^\lambda$.
Then Lemma \ref{ref-A.2} implies $\la\lambda, w{\ast}\chi^{\prime}\ra<\la\lambda,\chi'\ra=\la\lambda,\chi\ra$.

As
$({r}_{w{\ast}\chi^{\prime}},|{\bf S}_{w{\ast}\chi^{\prime}}|)=({r}_{\chi'},|{\bf S}_{\chi'}|)=({r}_{\chi},|{\bf S}_{\chi}|)$
we deduce $(r_{\chi},{\bf S}_{\chi})\not\prec({ r}_{w{\ast}\chi^{\prime}},{\bf S}_{w{\ast}\chi^{\prime}})$.
Then
  property \eqref{ena} in Lemma \ref{ref-1.5bis} implies $\la\lambda,\chi\ra\le \la\lambda, w{\ast}\chi^{\prime}\ra$.
This is a contradiction.
\end{proof}

\subsection{$G/G_e$-action on $X(T)^+$}
Here we allow $G$ to be non-connected. $W$ is still a
$G$-representation. We apply the previous results with $G$ replaced by
$G_e$. In particular $T\subset B\subset G_e$.  As we have seen in
\S\ref{sec:definition} the group $G/G_e$ acts on $X(T)^+$ via
$\bar{g}(\chi)=\chi\circ \sigma^{-1}_{\bar{g}}$.

Since $G/G_e$ also acts on the
weights on $W$, it may be made to act on  $\{1,\ldots,d\}$ via $\bar{g}(i)=j$ if $\beta_j=\beta_i\circ \sigma^{-1}_{\bar{g}}$.
This action extends to an action of $G/G_e$ on the partially
ordered set $(\RR^+\times \cal P(\{1,\dots,d\})^{3},\prec)$ introduced above.
\begin{lemmas}\label{ref-lemma-6.7}
The map
\begin{equation}
\label{eq:equivariant}
X(T)^+\mapsto \RR^+\times \cal P(\{1,\dots,d\})^{3}:\chi\mapsto (r_\chi,{\bf S}_\chi)
\end{equation}
is $G/G_e$-equivariant.
Moreover if $\chi\in X(T)^+$, $\bar{g}\in G/G_e$ and $\lambda\in Y(T)^-$ is as  in Lemma \ref{ref-1.5}\eqref{tri},    then $\sigma_{\bar{g}} \circ{\lambda}\in
  Y(T)^-$ satisfies property \eqref{tri} in Lemma \ref{ref-1.5} for
  ${{\chi\circ\sigma_{\bar{g}}}}$.  Consequently,
  $(T_{\sigma_{\bar{g}}\circ\lambda}^{\pm},T_{\sigma_{\bar{g}}\circ\lambda}^0)=\bar{g}^{-1}(T_{\lambda}^{\pm},T_{\lambda}^0)$.
\end{lemmas}

\begin{proof}
Let $\chi\in X(T)^+$. As usual we may assume we are in the non-trivial case.
  We obtain an expression of
  $\chi\circ\sigma_{\bar{g}}$ of the form \eqref{expresschi} with $\beta_i$ replaced by
  $\beta_i\circ \sigma_{g}$. Since this expression in minimal for
  $\chi\circ \sigma_{\bar{g}}$ (as otherwise applying $-\circ\sigma_{\bar{g}}^{-1}$ would give a smaller
  expression for $\chi$), $(r_{\chi\circ\sigma_{\bar{g}}},{\bf
    S}_{\chi\circ\sigma_{\bar{g}}})=\bar{g}^{-1}(r_{\chi},{\bf S}_{\chi})$.

  As $\la\lambda,\beta\circ
  \sigma_{\bar{g}}\ra=\la\sigma_{\bar{g}}\circ\lambda,\beta\ra$ for
  all $\beta\in X(T)$, it easy to verify $\sigma_{\bar{g}}
  \circ{\lambda}\in Y(T)^-_\RR$ satisfies property \eqref{tri} in
  Lemma \ref{ref-1.5} for ${{\chi\circ \sigma_{\bar{g}}}}$. This implies in particular
the last assertion of the lemma.
\end{proof}
\subsection{Reduction settings}\label{order}
Here we allow $G$ to be again non-connected.
\begin{lemmas}
\label{lem:total}
There exists a total ordering $<$ on $\RR^+\times \cal P(\{1,\dots,d\})^{3}$
such that the following conditions hold.
\begin{enumerate}
\item \label{ena1}
 If
$({ r},|{\bf S}|)<({ r}',|{\bf S}'|)$  then $({ r},{\bf S})<({ r}',{\bf S}')$.
\item
\label{dva1} If $({ r},{\bf S})<({ r}',{\bf S}')$ and  $({ r},{\bf S})$
and $({ r}',{\bf S}')$ are in different $G/G_e$-orbits then
$\bar{g}({ r},{\bf S})<\bar{h}({ r}',{\bf S}')$ for all $\bar{g},\bar{h}$ in $G/G_e$.
\end{enumerate}
\end{lemmas}
\begin{proof} The map \eqref{eq:order} is $G/G_e$-equivariant. We
choose an arbitrary totally ordering on the fibers of \eqref{eq:order} compatible with condition \eqref{dva1}. Combining this with~\eqref{ena1} completely fixes $<$.
\end{proof}
\begin{remarks}
It is clear that any total ordering
$<$ as in Lemma \ref{lem:total}\eqref{ena1}
refines the partial ordering
$\prec$.
\end{remarks}
\begin{lemmadefinitions}
\label{def:lchi}
It is possible to choose
for any $\chi\in X(T)^+$ a $\lambda_\chi\in Y(T)^-$ such that the following conditions are satisfied
\begin{enumerate}
\item \label{one1} $\lambda=\lambda_\chi$
  satisfies the property (\ref{tri})  in Lemma \ref{ref-1.5}.
\item \label{two1}
If $({r}_{\chi'},{\bf
  S}_{\chi'})=({r}_{\chi},{\bf S}_\chi)$ then $\lambda_\chi=\lambda_{\chi'}$.
\item \label{three1}
We have
  $\lambda_{\chi\circ \sigma_{\bar{g}}}=\sigma_{\bar{g}}\circ\lambda_\chi$ for all $\bar{g}\in G/G_e$.
\item \label{four1}
If $\bar{g}(T_{\lambda_\chi}^{\pm},T_{\lambda_\chi}^0)=(T_{\lambda_{\chi}}^{\pm},T_{\lambda_{\chi}}^0)$ then
$\sigma_{\bar{g}}\circ \lambda_\chi=\lambda_\chi$.
\end{enumerate}
\end{lemmadefinitions}
\begin{proof}
Choose representatives $(r_{\chi_i},{\bf S}_{\chi_i})$ for the orbits of the $G/G_e$-action on the image of
\eqref{eq:equivariant}. For each $i$ choose $\lambda'_i\in Y(T)^-$ such that
$S_{\chi_i}^{-}=T_{\lambda'_i}^-$, $S_{\chi_i}^+=T_{\lambda'_i}^+$, $S_{\chi_i}^0=T_{\lambda'_i}^0$ as in Lemma \ref{ref-1.5}\eqref{tri}.
Let $\bar{G}_i\subset G/G_e$ be the stabilizer of $(r_{\chi_i},{\bf S}_{\chi_i})$ and put
$\lambda_i=\sum_{\bar{g}\in \bar{G}_i} \sigma_{\bar{g}}\circ  \lambda'_i$. Then it is easy to see
that we still have $S_{\chi_i}^{-}=T_{\lambda_i}^-$, $S_{\chi_i}^+=T_{\lambda_i}^+$, $S_{\chi_i}^0=T_{\lambda_i}^0$
and moreover $\bar{g}\circ\lambda_i=\lambda_i$ if $\bar{g}\in \bar{G}_i$.

Now for $\chi\in X(T)^+$ write $(r_\chi,{\bf S}_\chi)=\bar{h}(r_{\chi_i},{\bf S}_{\chi_i})$ for suitable $i$ and
$\bar{h}\in G/G_e$. Then we put $\lambda_\chi=\sigma_{\bar{h}}\circ \lambda_i$. It is clear that
this is well defined and has the requested properties.
\end{proof}
\begin{lemmas}
\label{eq:diff}
Let $(\lambda_\chi)_\chi$ be as in Lemma \ref{def:lchi}. We have for $\bar{g}\in G/G_e$:
\[
\langle \lambda_\chi,\bar{g}(\chi)\rangle \ge \langle \lambda_\chi,\chi\rangle
\]
with equality if and only if $\bar{g}\in (G/G_e)^{\lambda_\chi}$.
\end{lemmas}
\begin{proof} By \eqref{eq:equivariant} we have
\begin{equation}
\label{eq:stab}
(r_{\bar{g}(\chi)},{\bf S}_{\bar{g}(\chi)})=\bar{g}(r_{\chi},{\bf S}_{\chi}).
\end{equation}
Hence in particular
\[
(r_{\bar{g}(\chi)},|{\bf S}_{\bar{g}(\chi)}|)=(r_{\chi},|{\bf S}_{\chi}|)
\]
and so $(r_\chi,S_\chi)\not\prec (r_{\bar{g}(\chi)},S_{\bar{g}(\chi)})$. It follows from Lemma \ref{ref-1.5bis}\eqref{dva} that
\[
\langle \lambda_\chi,\chi\rangle\le \langle \lambda_\chi,\bar{g}(\chi)\rangle.
\]
Also by Lemma \ref{ref-1.5bis} equality will happen precisely when $(r_{\bar{g}(\chi)},{\bf S}_{\bar{g}(\chi)})=
(r_{\chi},{\bf S}_{\chi})$ which by \eqref{eq:stab} implies that $\bar{g}$ stabilizes
${\bf S}_{\chi}=(T^{+}_{\lambda_\chi},T^-_{\lambda_\chi},T^0_{\lambda_\chi})$. By Lemma \ref{def:lchi}\eqref{four1} this
implies $\bar{g}\in (G/G_e)^{\lambda_\chi}$.
\end{proof}
Below we fix $(\lambda_\chi)_\chi$ as in Lemma \ref{def:lchi} and
we choose a total ordering on
$\RR^+\times \cal P(\{1,\dots,d\})^{3}$ as in Lemma \ref{lem:total}.
We put the induced ordering on $I:=\{({ r}_\chi,{\bf S}_\chi)\mid \chi\in X(T)^+\}\subset \RR^+\times \cal P(\{1,\dots,d\})^{3}$. As a totally
ordered set we have~$I\cong\NN$.  For $i\in I$ we put 
$F_i=\{\chi\in
X(T)^+\mid ({ r}_\chi,{\bf S}_\chi)=i\}$. This gives a $G/G_e$-equivariant partition
\begin{equation}
\label{eq:partition}
X(T)^+=\coprod_{i\in I} F_i.
\end{equation}
In each $F_i$ we choose one
representative which we denote by $\chi_i$. We write
$\lambda_i=\lambda_{\chi_i}$, $r_i=r_{\chi_i}$, ${\bf S}_i={\bf S}_{\chi_i}$.
By our choice of $\lambda_\chi$ and the definition of $F_i$,  $(r_i,\lambda_i,{\bf S}_i)$
depends only on  $i\in I$ and not on the choice of $\chi_i\in F_i$.

For $j\in I$ we write
\begin{align*}
\Lscr_{<j}&=\bigcup_{i\in I,i<j}F_i\subset X(T)^+,\\
\Lscr_{\le j}&=\bigcup_{i\in I,i\le j}F_i\subset X(T)^+.
\end{align*}
Let $J\subset I$ be the minimal representatives for the orbits of the
action of $G/G_e$ on~$I$.   By the choice of $J$ and
property \eqref{dva1} in Lemma \ref{lem:total} the set $\{i\in I\mid
i<j\}$ is $G/G_e$-invariant if $j\in J$ and hence the same is true for $\Lscr_{<j}$.
\begin{corollarys}\label{linredsec}
Let  $j\in J$  with $r_j\ge 1$ and $\chi\in F_j$. Then $(G,B,T,X,\Lscr_{<j},\chi,\lambda_j)$ is a reduction setting.
\end{corollarys}
\begin{proof}
We may reduce to the case that $X$ is affine. Then by
  Theorem \ref{ref-3.1-13} and Proposition \ref{ref-prop-3.4}
we may assume $G=G_e$ and $X=W^\vee$.
  Thus we need to verify the assumptions of Proposition
  \ref{-iC}. They are satisfied by Corollary \ref{betaplus} due to the
  choice of $\lambda_j=\lambda_\chi$ for every $\chi\in F_j$.
\end{proof}

\section{Proofs of the semi-orthogonal decompositions}
\label{sec:proofs}
In this section we will prove Theorem \ref{nonl} and
 along the way we will also prove Proposition \ref{ref-1.3} and Corollary \ref{twist}.
We continue to use the notations introduced in the previous section.

For the proof of Theorem
\ref{nonl} we select a finite open affine covering $X\quot G=\bigcup_i U_i$
and put $X_i=\pi^{-1}(U_i)$. Thus
$X=\bigcup_i X_i$ for
$G$-equivariant affine $G$-varieties $X_i$. We choose a
$G$-representation $W$ such that $W^\vee$ has a $T$-stable point
together with a closed $G$-equivariant embedding $\coprod_i X_i\hookrightarrow
W^\vee$.  We use $W$ to construct a partition \eqref{eq:partition} of
$X(T)^+$ as in \S\ref{order}. The arguments below will be based on ``reduction to the affine case'', i.e.\ to one of the $X_i$.

\medskip

For $j\in I$ let $\Dscr_{<j}$, $\Dscr_{\le j}$ be the triangulated subcategories of $\Dscr({{X}}/G)$
locally classically generated by $P_{\Lscr_{<j}}$, $P_{\Lscr_{\le j}}$ as in \S\ref{sec:localgen}
and put $\Lambda_{<j}=\pi_{s\ast}\uEnd_{X/G}(P_{\Lscr_{<j}})$.

For $j\in J$ let $\Dscr_j$ be the triangulated subcategory of $\Dscr({{X}}/G)$ locally classically generated by $\la
\RInd^{G}_{G^{\lambda_j,+}} (V_{G^{\lambda_j}}(\chi) \otimes
\Oscr_{X^{\lambda_j,+}})\mid \chi\in F_j\ra$.

 For a
$\cW_{G_e^\lambda}$-invariant $\nu\in X(T)_\RR$ we put
 \begin{align*}
{\Lscr_{r,\lambda,\nu}}&=X(T)^\lambda\cap (\nu-\bar{\rho}_{\lambda}+r\Sigma^0_\lambda),\\
U_{r,\lambda,\nu}&=\bigoplus_{\mu\in \Lscr_{r,\lambda,\nu}} \Ind^{\open^\lambda}_{G_e^\lambda} V_{G_e^\lambda}(\mu),\\
\Lambda_{r,\lambda,\nu}&=(\End(U_{r,\lambda,\nu})\otimes_k \Oscr_{X^\lambda})^{\open^\lambda}.
\end{align*}

Proposition \ref{ref-1.3} and
 Theorem \ref{nonl} will be  consequences of the following proposition.
\begin{proposition}\label{affineso}
\begin{enumerate}
\item \label{0}
If $j\not\in J$ then $\Dscr_{<j}=\Dscr_{\le j}$.
\end{enumerate}
Assume that $j\in J$ is such that $r_j\ge 1$. Then
\begin{enumerate}
\setcounter{enumi}{1}
\item\label{1}
$\D_{<j}\cong\Dscr(\Lambda_{<j})$  and $\Lambda_{<j}$  has finite global dimension when restricted to affine opens in $X\quot G$.
\item\label{3} $\D_{j}\cong \Dscr(\Lambda_{r_j,\lambda_j,\nu_j})$
  for $\nu_j=-\bar\rho+\bar\rho_{\lambda_j}+(\chi_j)_p$ (where $(\chi_j)_p$ was introduced in
\S\ref{sec:interms}) and
  $\Lambda_{r_j,\lambda_j,\nu_j}$  has finite global dimension when restricted to affine opens in $X\quot G$.
\item\label{2}
$\D_{\le j}=\la \D_{j},\D_{<j}\ra$ is a semi-orthogonal decomposition of $\D_{\le j}$.
\item\label{4} One has $\Dscr(X/G)=\bigcup_{j\in J}\Dscr_j$.
\end{enumerate}
\end{proposition}
\begin{proof}
\begin{enumerate}
\item[\eqref{0}]
We claim that $\bigoplus_{\chi\in \Lscr_{<j}}V(\chi)=\bigoplus_{\chi\in \Lscr_{\le j}}
V(\chi)$.  We only
have to prove that if $\chi\in F_j$ then $V(\chi)$ is a summand of
$\bigoplus_{\chi\in \Lscr_{<j}}
V(\chi)$. Since $j\not\in J$
there is $\bar{g}\in G/G_e$ such that $\bar{g}(j)<j$. Put $\chi'=\bar{g}(\chi)\in F_{\bar{g}(j)}\subset \Lscr_{<j}$.
Using \eqref{eq:restriction}, \eqref{eq:restriction1} we obtain
\begin{align*}
V(\chi')&=\Ind_{G_e}^G V_{G_e}(\chi')\\
&=\Ind_{G_e}^G {}_{\sigma^{-1}_{\bar{g}}}V_{G_e}(\chi)\\
&\cong \Ind_{G_e}^G V_{G_e}(\chi)\\
&=V(\chi)\,.
\end{align*}
\item[\eqref{1}]
The fact that  $\Dscr_{<j}=\Dscr(\Lambda_{<j})$ follows from Lemma \ref{lem:ff}. To prove
that $\Lambda_{<j}$ is locally of finite global dimension we
 may restrict to the case that $X$ is affine.
To prove that $\gldim \Lambda_{<j}<\infty$, by \cite[Thm. 4.3.1, Lem.\ 4.5.1]{SVdB} it suffices to consider the case
  ${{X}}=W^\vee$ and $G=G_e$.  We denote
  $\tilde{P}_{{\Lscr_{<j}},\chi}=\Hom_{{{X/G}}}(P_{<j},P_\chi)$.  By
  \cite[Lem.\ 11.1.1]{SVdB} it is enough to show that $\pdim \tilde
  P_{\cal L_{<j},\chi}<\infty$ for every $\chi\in X(T)^+$. Assume that
  there exists~$\chi$ such that $\pdim \tilde P_{\cal
    L_{<j},\chi}=\infty$ and take $\chi\in X(T)^+$ with minimal $({
    r}_\chi,|{\bf S}_\chi|)$.  Then $({ r}_\chi,{\bf S}_\chi)\not\leq({
    r}_\mu,{\bf S}_\mu)$ for all $\mu\in \cL_{<j}$ (for otherwise $\chi\in
  \cL_{<j}$ and hence $\pdim \tilde P_{\cal L_{<j},\chi}=0$).  Let
  $\lambda=\lambda_\chi$.  It follows from Lemma \ref{ref-1.5bis}
  (\ref{ena}) that $\la\lambda,\chi\ra<\la\lambda,\mu\ra $ for all
  $\mu\in\cL_{<j}$.  Thus, $C_{\cL_{<j},\lambda,\chi}:=\Hom_{X/G}(P_{\Lscr_{<j}},C_{\lambda,\chi})$ is acylic by
  \eqref{quasiiso} and the fact that the $\lambda$-weights of $k[{{X}}^{\lambda,+}]$ are $\le 0$ (see \S\ref{ref-1.2-0}). We have
 $({ r}_{\chi'},|{\bf S}_{\chi'}|)<({
    r}_\chi,|{\bf S}_\chi|)$ for all $\tilde P_{\cal L_{<j},\chi'}\neq
  \tilde P_{\cal L_{<j},\chi}$ that appear in $C_{\cL_i,\lambda,\chi}$ by
  \eqref{eq:weights} and Corollary \ref{betaplus}.  Hence
  $\pdim\tilde P_{\cal L_{<j},\chi'}<\infty$ by the minimality
  assumption, and therefore $\pdim\tilde P_{\cal L_{<j},\chi}<\infty$, a
  contradiction.
\item[\eqref{3}]
Now we use the fact that $(G,B,T,{{X}},\Lscr_j,\chi,\lambda_j)$ is a reduction setting for $\chi\in F_j$
 by Corollary \ref{linredsec}.
Let us abbreviate $D_{j,\chi}=\RInd^{G}_{G^{\lambda_j,+}} (V_{G^{\lambda_j}}(\chi) \otimes_k  \Oscr_{X^{\lambda_j,+}})$.
By Proposition \ref{Di} we have
\[
\pi_{s\ast}\uREnd_{X/G}(\oplus_{\chi\in F_j} D_{{j},\chi})=
\pi_{s\ast}\uEnd_{X^\lambda/\open^\lambda}(\oplus_{\chi\in F_j} P_{\open^\lambda,\chi})
\]
 and the latter is equal to  $\Lambda_{r_j,\lambda_j,\nu_j}$
by Lemma \ref{chi_lambda}.

To prove that $\Lambda_{r_j,\lambda_j,\nu_j}$ locally has finite global dimension we may reduce to
the affine case. Then by \cite[Thm. 4.3.1, Lem.\ 4.5.1]{SVdB} we may reduce to the case $G=G_e$ and $X=W^\vee$. Finally we invoke  Proposition \ref{sigma}.
\item[\eqref{2}]
The fact that $\Dscr_{\le j}$ is generated by $\Dscr_{<j}$ and $\Dscr_j$ follows from \eqref{ref-2.3-8}. The
fact that $\Hom_{X/G}(\Dscr_{<j},\Dscr_j)=0$ follows from  \eqref{ref-2.1-6}
by a suitable version of the  local global spectral
sequence on $X\quot G$.
\item[\eqref{4}] This follows from Lemma \ref{lem:thick}.
\end{enumerate}
\def\qed{}\end{proof}
\begin{proof}[Proof of Theorem \ref{nonl}]
Put
\[
j_0=\min \{j\in J\mid r_j\ge 1\}.
\]
Then we have by Lemma \ref{lem:ff}
$
\Dscr_{\le j_0}=\Dscr(\Lambda_{1,0,0})
$. Now write
\[
\{j\in J\mid r_j\ge 1\}=\{j_0,j_{1},j_{2},\ldots\}.
\]
Put $\Dscr_0=\Dscr_{\le j_0}$ and for $i>0$ let $\Dscr_{-i}$
be the right orthogonal of $\Dscr_{\le j_{i-1}}
=\Dscr_{< j_{i}}$ in $\Dscr_{\Lscr_{\le j_{i}}}$. By Proposition \ref{affineso}
(\ref{0},\ref{2},\ref{4}) we have a semi-orthogonal decomposition
\begin{equation}
\label{eq:semi}
\Dscr(X/G)=\langle \ldots,\Dscr_{-2},\Dscr_{-1},\Dscr_0\rangle
\end{equation}
and by  Proposition \ref{affineso}\eqref{3} each of the $\Dscr_{-1}$, $\Dscr_{-2}$,\dots has the required form. The corresponding statement for $\tilde{\Dscr}_{X/G}$ follows by replacing $X\quot G$ by open subschemes.
\end{proof}
\begin{remark} If follows from Lemma \ref{gldimadm} that  $\Dscr_{-j}\subset \Dscr(X/G)$ and
$\D_{\leq j}$ for $j\in J$ are admissible.
\end{remark}
 \begin{proof}[Proof of Proposition \ref{ref-1.3}]
This corresponds to the special case $X=W^\vee$ in the proof of  Theorem \ref{nonl}.
\end{proof}

\section{The case that $X$ does not have a $T$-stable point.}\label{introsec:nonstable}
\label{sec:nonstable}
In this section we will assume throughout that $G$ is connected and
that $X$ is a connected smooth $G$-variety such that a good
quotient $X\quot G$ exists.  We will give an alternative
semi-orthogonal decomposition of $\D(X/G)$ in case $X$ does not have a
$T$-stable point.
\begin{lemma}
\label{prop:two} Assume that $X$ does not have a $T$-stable point. Then
 at least one of the following settings holds:
\begin{enumerate}
\item \label{sit1} There is a non-trivial normal connected subgroup $K$ of $G$ acting trivially on $X$.
\item \label{sit2} There is a non-trivial central one parameter group $\nu:G_m\r Z(G)$
such that $X=X^{\nu,+}$.
\end{enumerate}
\end{lemma}
\begin{proof}
Since
$X$ does not have a $T$-stable point we have
\begin{equation}
\label{eq:union}
X=\bigcup_{\sigma\in X(T)-\{0\}} X^{\sigma,+}\,.
\end{equation}
It is well-known and easy to see that there are only a finite
number of distinct $X^{\sigma,+}$.  Indeed: by covering $X\quot G$ by a finite number
of affines, it suffices to verify this in the affine
case and then it follows by embedding~$X$ into a representation. Hence since
the union in \eqref{eq:union} is finite and $X$ is irreducible we
find  that there is some $\sigma\in X(T)-\{0\}$ such that $X=X^{\sigma,+}$.

It follows that $X=X^{w\sigma,+}$ for every $w\in \Wscr$.
Put $\nu=\sum_{w\in \Wscr} w\sigma$.
Then $\nu$ is a~$\W$-in\-variant $1$-parameter subgroup of $T$ and in particular its image is contained in the center of $G$.

We claim $X^{\nu,+}=X$, $X^{\nu}\subset X^\sigma$. We may check this in the case that $X$ is affine. Let $C\subset X(T)_\RR$ be the cone
spanned by the weights of  $k[X]$.  Since $X^{\sigma,+}=X$ we have $\langle \sigma,C\rangle\le 0$. Since $C$ is $\Wscr$-invariant
we immediately deduce $\langle \nu,C\rangle \le 0$ and hence $X^{\nu,+}=X$.

To prove $X^{\nu}\subset X^\sigma$ we have to verify that if $\chi\in C$ and $\langle \nu,\chi\rangle=0$ then
$\langle \sigma,\chi\rangle=0$. To prove this it suffices to observe that $\langle \nu,\chi\rangle=\sum_{w\in \Wscr} \langle \sigma,w^{-1}\chi\rangle$
and this can only be zero if $ \langle \sigma,w^{-1}\chi\rangle=0$ for all $w\in \Wscr$.

If $\nu\neq 0$ then we are in situation \eqref{sit2}. If $\nu=0$ then
$X^{\nu}=X$. Hence also $X^\sigma=X$ by the above discussion. Therefore \eqref{sit1} holds
with $K$ being the identity component of $\ker(G\r \Aut(X))$. The group $K$ is
not trivial as it contains $\im \sigma$.
\end{proof}
To continue it will be convenient to slightly generalize our setting in a similar way as \cite[Thm 1.6.3]{SVdB}.
We will however use different notations which are more adapted to the current setting.
We will assume that $G$ contains a finite central subgroup~$A$ acting trivially on $X$ with $\bar{G}:=G/A$.
Let $X(A):=\Hom(A,G_m)$ be the character group of $A$.
For $\tau\in X(A)$ let ${\D}(X/G)_{\tau}$
be the triangulated subcategory of ${\D}(X/G)$ consisting of complexes
on which $A$ acts as $\tau$.
We have an orthogonal decomposition
\begin{equation}
\label{eq:orthogonal}
\D(X/G)=\bigoplus_{\tau\in X(A)} {\D}(X/G)_{\tau}
\end{equation}
and moreover $ {\D}(X/G)_{0}=\D(X/\bar{G})$. In general we should think of
${\D}(X/G)_{\tau}$ as a twisted version of $\D(X/\bar{G})$.
\begin{proposition}\label{twist}
Let $\tau\in X(A)$. Then there exists a semi-orthogonal decomposition of
${\D}(X/G)_{\tau}$ of the form described in Theorem \ref{nonl}.
\end{proposition}
\begin{proof}
Let the notation be as in \S\ref{sec:proofs}. We have $A\subset T$. Let $\bar{T}=T/A$. Then
there is an exact sequence
\begin{equation}\label{eq:barTTA}
0\r X(\bar{T})\r X(T)\r X(A)\r 0.
\end{equation}
Let $X(\bar{T})_\tau\subset X(T)$ be the inverse image of $\tau\in X(A)$.
Let $\chi\in F_j\cap X(\bar{T})_{\tau}$.
Since~$A$ acts trivially on $X$, it acts with the character $\tau$ on the right-hand side of \eqref{ref-2.3-8} with $\Lscr=\Lscr_{<j}$.
Since $\pi_{s\ast}\uHom_{X/{G}}(P_{{G},\chi'},P_{{G},\chi})=0$ if $\chi'\not\in X(\bar{T})_{\tau}$,  \eqref{ref-2.3-8} is thus still an isomorphism if we replace $\Lscr_{<j}$ by
$\Lscr_{<j,\tau}=\Lscr_{<j}\cap {X(\bar{T})}_{\tau}$.
Put
\begin{align*}
\Lambda_{<j,{\tau}}&=\pi_{s*}\uEnd_{X/{G}}(P_{\Lscr_{<j,\tau}}),\\
\Lscr_{j,{\tau}}&=X(\bar{T})_{\tau}^{\lambda_j}\cap (\nu-\bar{\rho}_{\lambda_j}+r_j\Sigma^0_{\lambda_j}),\\
U_{j,{\tau}}&=\bigoplus_{\mu\in \Lscr_{j,{\tau}}} \Ind^{\open^{\lambda_j}}_{{G}_e^{\lambda_j}} V_{{G}_e^{\lambda_j}}(\mu),\\
\Lambda_{j,{\tau}}&=(\End(U_{j,{\tau}})\otimes_k \Oscr_{X^{\lambda_j}})^{\open^{\lambda_j}}.
\end{align*}
We have $\Lambda_{<j}=\oplus_{ \tau\in X(A)} \Lambda_{<j,{\tau}}$,
$\Lambda_{j}=\oplus_{\tau\in X(A)} \Lambda_{j,{\tau}}$.
As $\Lambda_{<j}$ and $\Lambda_j$ have finite global dimension when restricted to affine opens in $X\quot{G}$, the same holds for
$\Lambda_{<j,{\tau}}$, $\Lambda_{j,{\tau}}$.
Arguing as above we thus obtain a semi-orthogonal decomposition $\D_{\leq j,{\tau}}=\la \D_{j,{\tau}},\D_{<j,{\tau}}\ra$ (with the obvious definitions for $\D_{j,{\tau}},\D_{<j,{\tau}}$),
$\D_{<j,{\tau}}\cong \D(\Lambda_{<j,{\tau}})$, $\D_{j,{\tau}}\cong \D(\Lambda_{j,{\tau}})$. Then the proof continues as before.
\end{proof}
If $K$ is a connected normal subgroup of $G$ then we will define a
\emph{pseudo-comple\-ment} of $K$ as a connected normal subgroup $Q$ of $G$ such that
$K$ and $Q$ commute, $G=KQ$ and $K\cap Q$ is finite. It follows easily
from \cite[Theorem 8.1.5, Corollary 8.1.6]{Springer}  that such a
pseudo-complement always exists.
\begin{proposition}
\label{prop:nonstable1}
Assume $X$ does not have a $T$-stable point and
that we are in the situation of Lemma \ref{prop:two}\eqref{sit1}. Let $Q$ be a  pseudo-complement of $K$ in $G$. Then there is a finite central subgroup $A_Q$ of $Q$ acting trivially on $X$ such that there is  an orthogonal decomposition
\begin{equation}
\label{eq:Q}
\Dscr(X/G)_{\tau}\cong \bigoplus_{i\in I}  \Dscr(X/Q)_{\mu_i}
\end{equation}
for a suitable collection of $(\mu_i)_{i\in I}\in X(A_Q)$.
\end{proposition}
\begin{proof}
Let $\tilde{G}:=K\times Q\r G$ be the multiplication map and let $\tilde{A}\subset K\times Q$ be
the inverse image of $A$. Let $A_K\subset K$, $A_Q\subset Q$ be the images of $\tilde{A}$ under the
projections $K\times Q\r K,Q$.
It is easy to see that $A_Q$ acts trivially on $X$.
Let $\tilde{\tau}$ be the composition $\tilde{A}\r A\xrightarrow{\tau} G_m$. In a similar way,
if $\mu\in X(A_K)$ or $\mu\in X(A_Q)$ then we denote by $\tilde{\mu}$ the
element of $X(\tilde{A})$, obtained by composing $\mu$ with the appropriate projections $\tilde{A}\r A_K,A_Q$.

We have
\begin{equation}
\label{eq:part1}
D(X/G)_{\tau}=D(X/\tilde{G})_{\tilde{\tau}}\,.
\end{equation}
Let $\Irr(K)$ be the set of isomorphism classes of irreducible representations of $K$.
We have an orthogonal decomposition
\begin{equation}
\label{eq:part2}
\bigoplus_{V\in \Irr(K)} D(X/Q) \r D(X/\tilde{G}):(\Fscr_V)_V\mapsto \oplus_{V\in\Irr(K)} V\otimes_k \Fscr_V\,.
\end{equation}
If $V\in\Irr(K)$ then $A_K$ acts on $V$ via a character which we
denote by $\chi_V\in X(A_K)$.
Combining \eqref{eq:part1} and \eqref{eq:part2} yields
\[
D(X/G)_{\tau}\cong\bigoplus_{(\mu,V)\in X(A_Q)\times \Irr(K), \tilde{\mu}+\tilde{\chi}_V=\tilde{\tau}} D(X/Q)_{\mu}
\]
which implies \eqref{eq:Q}.
\end{proof}
\begin{proposition}
\label{prop:nonstable2}
Assume $X$ does not have a $T$-stable point and
that we are in the situation of Lemma \ref{prop:two}\eqref{sit2}. Let $Q$ be a pseudo-complement of $\im \nu$ in $G$. Then there is a finite central subgroup $A_Q$ of $Q$ acting trivially on $X^\nu$ such that there is  a semi-orthogonal decomposition
\begin{equation}
\label{eq:Q1}
\Dscr(X/G)_{\tau}= \langle \Dscr(X^\nu/Q)_{\mu_i}\mid i\in I\rangle
\end{equation}
for a suitable totally ordered set $I$ and a
collection of $(\mu_i)_{i\in I}\in X(A_Q)$.
\end{proposition}
\begin{proof}
 Without loss of generality we may, and we will, assume that $\nu$ is injective. We put $K=\im \nu\cong G_m$ and we borrow the associated notation from the proof of Proposition \ref{prop:nonstable1}. One checks that
in this case $A_Q$ acts indeed trivially on $X^\nu$.
The set $\Irr(K)$ is equal to $\{(\chi_n)_n\}$ where $\chi_n\in X(K)$ is such that $\chi_n(z)=z^n$.

\medskip

We know by Lemma \ref{lem:thick}
that $\Dscr(X/\tilde{G})$ is locally classically generated by $P_{n,V}:=(\chi_n\otimes_k V\otimes_k \Oscr_X)_{n,V}$, with $n\in \ZZ$, $V\in \Irr(Q)$.
We also put  $P_{V}^\nu:=V\otimes_k \Oscr_{X^\nu}\in \Dscr(X^\nu/Q)$.
We claim that  for $n\le m$ one has
\begin{equation}
\label{eq:dndef}
\pi_{s\ast} \uRHom_{X/\tilde{G}}(P_{n,V},P_{m,V'})=
\begin{cases}
0&\text{if $n<m$}\\
 \pi_{s,\ast} j_{s,\ast}\uRHom_{X^\nu/Q}(P^\nu_{V},P^\nu_{V'})&\text{if $n=m$}\\
\end{cases}
\end{equation}
where $j:X^\nu\r X$ is the embedding and $j_s:X^\nu/Q\r X/\tilde{G}$ is the corresponding map of quotient stacks.

 To prove this we may reduce to the case that $X\quot \tilde{G}$ is affine. Then as usual $\nu$ induces a grading on $k[X]$
which, as $\nu$ is central, is compatible
with the $G$-action.

In the affine case $\pi_{s,\ast}  \uRHom_{X/\tilde{G}}(P_{n,V},P_{m,V'})$ is the
quasi-coherent sheaf on $X\quot G$ associated to the $k[X]^G$-module given by $(\Hom(V,V')\otimes_k k[X]_{m-n})^{\tilde{G}}$ which is zero if $n<m$
since the hypothesis $X=X^{\nu,+}$ implies that
the grading on $k[X]$ is concentrated in negative degree. Similarly if $n=m$ then we have $(\Hom(V,V')\otimes_k k[X]_{m-n})^{\tilde{G}}=(\Hom(V,V')\otimes_k k[X^\nu])^{Q}$
finishing the proof of \eqref{eq:dndef}.

\medskip

Let $\Dscr_n\subset \Dscr(X/\tilde{G})$ be locally classically generated by $(P_{-n,V})_{V\in \Irr(Q)}$. Then using Proposition \ref{th:recognition} and \eqref{eq:dndef}
we see that we have a semi-orthogonal decomposition
\begin{equation}
\label{eq:part11}
\Dscr(X/\tilde{G})=\langle \Dscr_n\mid n\in \ZZ\rangle\,.
\end{equation}
The next step is to describe the $\Dscr_n$.
We claim that there is an equivalence of categories
\begin{equation}
\label{eq:part12}
\Dscr_n\r D(X^\nu/Q):F\mapsto \chi_{n} \otimes_k Lj_{s}^\ast( F)\,.
\end{equation}
Let $F,F'\in \Dscr_n$. We have to prove that the natural map
\[
\RHom_{X/\tilde{G}}(F,F')\r \RHom_{X^\nu/Q}(\chi_{n} \otimes_k Lj_{s}^\ast(\c F),\chi_{n} \otimes_k Lj_{s}^\ast( F'))
\]
is an isomorphism.  Using the local global spectral sequence it suffices to prove that
\[
\pi_{s,\ast }\uRHom_{X/\tilde{G}}(F,F')\r \pi_{s,\ast} Rj_{s,\ast} \uRHom_{X^\nu/Q}(\chi_{n} \otimes_k Lj_{s}^\ast( F),\chi_{n} \otimes_k Lj_{s}^\ast( F'))
\]
is an isomorphism. To do this we may assume that $X\quot \tilde{G}$ is
affine. Then we can check it on the generators $P_{-n,V}$ of $\Dscr_n$
and finally we invoke \ref{eq:dndef}.

\medskip

Combining the equivalence \eqref{eq:part12} with the semi-orthogonal decomposition \eqref{eq:part11}
we obtain a a semi-orthogonal decomposition
\[
\Dscr(X/\tilde{G})=\langle \Dscr(X^\nu/Q)\mid n\in \ZZ\rangle
\]
Considering suitable subset of the local generators one obtains
in the same way a semi-orthogonal decomposition of $\Dscr(X/\tilde{G})_{\tilde{\tau}}$. To be more precise
consider the following set
\[
\Sscr=\{(n,\mu)\in \ZZ\times A_Q\mid -\tilde{\chi}_n+\tilde{\mu}=\tilde{\tau}\}
\]
Let $\prec$ be the partial ordering on $\Sscr$ induced from the projection $\Sscr\r \ZZ$, i.e. $(n,\mu)\prec (n',\mu')$
if and only if $n<n'$.
Let $<$ be a total ordering on $\Sscr$ which refines $\prec$. Then we have a semi-orthogonal
decomposition
\[
\Dscr(X/\tilde{G})_{\tilde{\tau}}=\langle \Dscr(X^\nu/Q)_{\mu}\mid (n,\mu)\in \Sscr\rangle\,.
\]
Combining this with the identification \eqref{eq:part1} yields \eqref{eq:Q1}.
\end{proof}
\begin{remark} Even if $A=0$ (and hence $\Dscr(X/G)_{\tau}=\Dscr(X/G)$), the group $A_Q$ and the twisting characters $\mu_i$ will generally be non-trivial in Propositions
\ref{prop:nonstable1},\ref{prop:nonstable2}.
\end{remark}
\begin{remark}
  Note that in Propositions
  \ref{prop:nonstable1},\ref{prop:nonstable2} we have $\dim Q<\dim
  G$. Thus we have made genuine progress.
By repeatedly applying
  Propositions \ref{prop:nonstable1},\ref{prop:nonstable2} we reduce to a semi-orthogonal decomposition of $\Dscr(X/G)_\tau$ involving a set of $\Dscr(X'/G')_{\tau'}$
such that $X'$ has a $T'$-stable point
for $T'$ a maximal torus of $G'$,
thus justifying Remark \ref{rem:notTstable} (and also making it more precise).
\end{remark}

\section{The quasi-symmetric case}\label{appA}
In this section we refine the semi-orthogonal decomposition of $\D(X/G)$ given in Proposition \ref{ref-1.3} in the quasi-symmetric case. In some cases the refined decomposition consists of (twisted) non-commutative crepant resolutions of certain quotient singularities for reductive groups.

\subsection{Main result}
Let $W$ be a finite dimensional $G$-representation of dimension $d$ such that $X=W^\vee$. Let $(\beta_i)_{i=1}^d\in X(T)$ be the $T$-weights of $W$. \emph{Throughout this section we assume that $W$ is quasi-symmetric}; i.e., for every line $\ell\subset X(T)_\RR$ through
the origin we have $\sum_{\beta_i\in\ell}\beta_i=0$.

Let $\Delta\subset\RR^n$ be a bounded closed convex polygon.
For $\varepsilon\in \RR^n$ parallel to the linear space spanned by $\Delta$ put
\[
\begin{aligned}
\Delta_\varepsilon&=\bigcup_{r>0} \Delta\cap (r\varepsilon+\Delta),\\
\Delta_{\pm\varepsilon}&=\Delta_\varepsilon\cap \Delta_{-\varepsilon}.
\end{aligned}
\]

We say that a $\Wscr$-invariant $\varepsilon\in X(T)_\RR$ is {\em generic} for $\Delta$ if it parallel to $\Delta$ but not parallel to any face of $\Delta$. Note that such $\varepsilon$ does not necessarily exist - for example there may be no non-zero $\Wscr$-invariant vectors at all. 
We shall say that $\varepsilon$ is {\em weakly generic} for $\Delta$ if it is parallel to $\Delta$ but not parallel to faces of $\Delta$ for which there exist non-parallel $\Wscr$-invariant vectors.

Let $0\neq \lambda\in Y(T)^-$ and  
let $A$ be a finite central subgroup of $G^\lambda$ which acts trivially on $X^\lambda=(W_\lambda)^\vee$. 
Fix $\tau\in A$ and $\Wscr_{G^\lambda}$-invariant $\varepsilon,\nu\in X(T)_\RR$. 

 We denote $\bar{T}=T/A$ and $X(\bar{T})_\tau$ the inverse image of $\tau\in X(A)$ under the natural projection map $X(T)\to X(A)$ (see \eqref{eq:barTTA}), and we write $X(\bar{T})^\lambda_\tau$  for the $G^\lambda$-dominant weights inside $X(\bar{T})_\tau$.

Put
 \begin{align*}
\Lscr_{\lambda,\nu,\tau}^{\varepsilon}&= X(\bar T)^\lambda_\tau\cap\left(\nu-\bar{\rho}_\lambda+
(1/2)(\bar{\Sigma}_\lambda)_{\varepsilon}\right),\\
U_{\lambda,\nu,\tau}^{\varepsilon}&=\bigoplus_{\mu\in \Lscr_{\lambda,\nu,\tau}^{\varepsilon}}V_{G^\lambda}(\mu),\\
\Lambda_{\lambda,\nu,\tau}^{\varepsilon}&=(\End U_{\lambda,\nu,\tau}^{\varepsilon}\otimes_k \Sym W_\lambda)^{G^\lambda}.
\end{align*}

As $\lambda\neq 0$ acts trivially on $X^\lambda$ the latter
does not have a $T=T_{G^\lambda}$-stable point.
Since $W_{\lambda}$
 is also quasi-symmetric,
we are in situation of Lemma \ref{prop:two}\eqref{sit1}.
We denote by $Q_\lambda \subset G^\lambda$ a pseudo-complement of the stabilizer subgroup ${\rm Stab}(X^\lambda)\subset G^\lambda$.
Recall the following definition from \cite{SVdB}:
\begin{definition}\label{def:generic}
We say that  $W$ is a {\em generic} $G$-representation if
\begin{enumerate}
\item  $X$ contains a point with closed orbit and trivial stabilizer.
\item If $X^{\mathbf{s}}\subset X$ is the locus of points that satisfy (1) then $\codim
  (X-\Xs,X) \ge 2$.
\end{enumerate}
We say that  $W$ is a {\em pseudo-generic} $G$-representation if the stabilizing subgroup ${\rm Stab}(X)$ of $G$ is finite, and $W$ is a generic $G/{\rm Stab}(X)$-representation.
\end{definition}

\begin{remark}\label{rmk:pseudo}
Let $W,U$ be finite dimensional $G$-representations. Assume that
$A={\rm Stab}W \subset G$ is finite. 
We have $(SW)^G=(SW)^{G/A}$, and $(U\otimes SW)^G\cong (U^A\otimes SW)^{G/A}$. 
As such the results stated in \cite{SVdB} for generic representations extend trivially to pseudo-generic representations. 
\end{remark}

We will need a ``crepant'' version of Proposition \ref{sigma}. For the definition of {twisted non-commutative resolution} (twisted NCCR) we refer to \cite[Definition 3.2]{SVdB}.
\begin{proposition}\cite[Theorems 1.6.3, 1.6.4]{SVdB}\label{nccr}
 We have $\gldim\Lambda_{\lambda,\nu,\tau}^{\varepsilon}<\infty$.
Moreover, if $W_\lambda$ is a pseudo-generic
 $Q_\lambda$-representation, $\Lscr_{\lambda,\nu,\tau}^{\varepsilon}\neq\emptyset$ and
\begin{equation}\label{prazno}
X(\bar T)_\tau^\lambda\cap \left(\nu-\bar\rho_\lambda
+(1/2)\left((\bar{\Sigma}_\lambda)_{\pm\varepsilon}- {\Sigma}_\lambda\right)\right)=\emptyset
\end{equation}
then $\Lambda_{\lambda,\nu,\tau}^{\varepsilon}$ is a twisted NCCR of $(\Sym W_\lambda)^{G^\lambda}$.
\end{proposition}

\begin{proof}
The first part follows by the proof of \cite[Theorem 1.6.3]{SVdB}.
To show that $\Lambda_{\lambda,\nu,\tau}^{\varepsilon}$ is Cohen-Macaulay we can proceed as in the proof of  \cite[Theorem 1.6.4]{SVdB} since $X^\lambda$ has a stable $Q_\lambda$-point (as $W_\lambda$ is a pseudo-generic $Q_\lambda$-representation). 
 Then $\Lambda_{\lambda,\nu,\tau}^{\varepsilon}$ is a twisted NCCR by Remark \ref{rmk:pseudo} and \cite[Proposition 4.1.6]{SVdB}.
\end{proof}

We can now state the main result of this appendix (see \S\ref{sec:quasimod} below for the proof).
\begin{proposition}\label{quasisod} 
Let $X$ have a $T$-stable point. 
There exist $\lambda_i\in Y(T)^-$,
finite central subgroups $A_i$ of $G^{\lambda_i}$ acting trivially on $X^{\lambda_i}$, $\tau_i\in X(A_i)$, $\W_{G^{\lambda_i}}$-invariant $\nu_i,\varepsilon_i\in X(T)_\RR$, such that $\D(X/G)$ has a semi-orthogonal decomposition $\D=\la\dots,\D_{-2},\D_{-1},\D_0\ra$ with $\D_{-i}\cong \D(\Lambda_{\lambda_i,\nu_i,\tau_i}^{\varepsilon_i})$. Moreover, we may assume $\D_0\cong \D(\Lambda_{0,0,0}^{\varepsilon_0})$, and that $\epsilon_i$ is  weakly generic for $\bar{\Sigma}_{\lambda_i}$. 
\end{proposition}

Combining Propositions \ref{nccr}, \ref{quasisod} we thus obtain under some favourable conditions a semi-orthogonal decomposition of $\D(X/G)$ consisting of twisted NCCRs.
\begin{corollary}\label{quasicor}
If for every $\lambda\in Y(T)^-$  either  $W_\lambda=\{0\}$ or $W_\lambda$ is a pseudo-generic $Q_\lambda$-representation and the condition \eqref{prazno} holds for all $\Wscr_{G^\lambda}$-invariant $\nu\in X(T)_\RR$ and for all $\Wscr_{G^\lambda}$-invariant $\varepsilon\in X(T)_\RR$ which are weakly generic for $\bar{\Sigma}_{\lambda}$, and all $\tau\in X(A)$ for an arbitrary finite central subgroup $A$ of $G^\lambda$, then $\D(X/G)$ has a semi-orthogonal decomposition
$\D=\la\dots,\D_{-2},\D_{-1},\D_0\ra$ with $\D_{-i}\cong \D(\Lambda_{\lambda_i,\nu_i,\tau_i}^{\epsilon_i})$, and $\Lambda_{\lambda_i,\nu_i,\tau_i}^{\epsilon_i}$ is a twisted NCCR of $(\Sym W_{\lambda_i})^{G^{\lambda_i}}$.
\end{corollary}

\subsection{Examples}
Here we list some examples of quotient singularities for reductive groups
for which Corollary \ref{quasicor}
gives a semi-orthogonal decomposition such that its components are NCCRs of singularities of the same type.

In the cases below one can verify \eqref{prazno} in a similar way as the analogous condition was verified in \cite{SVdB}.

\subsubsection{Torus action}
In the case of $G=T$  the condition  \eqref{prazno} holds for every generic $\epsilon\in X(T)_\RR$, and therefore also for every weakly generic $\epsilon$ since the two notions coincide in this case. Considering also the pseudo-genericity of $W_\lambda$'s we deduce  the following proposition. 
\begin{proposition}
 Assume that  on every line $\ell\subset X(T)_\RR$ through the origin on which lies a nonzero $\beta_i$ there lie at least two $\beta_i$ on each of its sides. Then the condition in the Corollary \ref{quasicor} is satisfied and thus $\Dscr(X/G)$ admits a semi-orthogonal decomposition consisting of NCCRs.
\end{proposition}

\subsubsection{$\rm SL_2$-action}
Let $G={\rm SL}_2={\rm SL}(V)$ for a $2$-dimensional vector space $V$. 
 Put $W=\bigoplus_{i=0}^{d_i} S^{d_i}V$, 
 $c=|\{i\mid d_i=0\}|$,
\[
s^{(n)}=\begin{cases}
n+(n-2)+\cdots+1=\dfrac{(n+1)^2}{4}&\text{if $n$ is odd}\\
n+(n-2)+\cdots+2=\dfrac{n(n+2)}{4}&\text{if $n$ is even}
\end{cases}
\]
and $s=\sum_i s^{(d_i)}$. Set $R=(\Sym W)^G$, $M=\oplus_{0\leq i\leq s/2-1} (\Sym(W) \otimes_k S^iV)^G$. 

\begin{proposition}
\begin{enumerate}
\item\label{A}
If $W$ is a sum of $k^c$ and one of the following representations
\begin{equation*}\label{A}
V,S^2V,V\oplus V,V\oplus S^2V, S^2V\oplus S^2V, S^3V, S^4V,
\end{equation*}
then $(\Sym W)^G$ is a polynomial ring, and $\D(X/G)$ has a semi-orthogonal decomposition $\D=\la\dots,\D_{-2},\D_{-1},\D_0\ra$, with $\D_{-i}\cong \D(k[x_1,\dots,x_{c}])$ for $i>0$, 
 $\D_{0}\cong\D(R)$.
 \item \label{B}
 If $W$ is not as in the case \eqref{A} 
 and if 
  $s$ is odd then $\Dscr(X/G)$ has a semi-orthogonal decomposition $\langle \dots, \D_{-2},\D_{-1},\D_0\ra$ with $\D_{-i}\cong \D(k[x_1,\dots,x_{c}])$ for $i>0$, 
 $\D_0\cong\D(\End_R(M))$.
 \end{enumerate}
\end{proposition}
 
 \begin{proof}
 We note that every representation of $\rm SL_2$ is quasi-symmetric. 
Let us first assume that we are in the case \eqref{B}. 
We have $\Sigma=]-s,s[$. 
We want to verify the condition \eqref{prazno} for $\lambda=0$. 
Note that $\epsilon=0$ and $A=0$. As $s$ is odd, 
 $W$ satisfies \eqref{prazno} (cf. \cite[Theorem 1.4.5]{SVdB} and its proof). 
Since $s$ is odd it also follows that $W$ is a generic ${\rm SL}_2$-representation. 
 Moreover if $\lambda\neq 0$ then $W_\lambda$ is a sum of trivial
representations and hence $\operatorname{Stab}(W_\lambda)=G^\lambda$. Thus $Q_\lambda$ is the
trivial group. So in particular $W_\lambda$ is $Q_\lambda$-generic for $\lambda\neq 0$. 
Since $Q_\lambda$ is the trivial group  for $\lambda\neq 0$ we verify that $\Lambda_{\lambda,\nu,\tau}^\varepsilon\cong \Sym W_\lambda\cong k[x_1,\dots,x_c]$ if $\Lambda_{\lambda,\nu,\tau}^\varepsilon\neq 0$.

 If $W$ is as in \eqref{A} then it is well-known that $(\Sym W)^G$ is a polynomial ring. Therefore $\langle P_0\rangle\cong \Dscr((\Sym W)^G)$ is admissible in $\D(X/G)$, and thus we can set it to be equal to $\Dscr_0$, and we can adjust the proof of Proposition \ref{affineso} to get as above $\Dscr_{-i}\cong \langle \Sym W_\lambda\rangle \cong k[x_1,\dots,x_c]$ for $\lambda\neq 0$. 
 \end{proof}

\subsubsection{Determinantal varieties}
Let $n<h$, $W=(V^*)^h\oplus V^h$, $\dim V=n$, $G={\rm GL}(V)$, and let $Y_{n,h}=W\quot G$ be the variety of $h\times h$-matrices of rank $\leq n$.  We denote by $\Lambda_j$ the NCCR of $Y_{j,h}$ given by \cite[Proposition 5.2.2]{SVdB} (constructed earlier in \cite{BLVdB2,SegalDonovan}). For convenience we set $\Lambda_0:=k$.

\begin{proposition}
Let $\Dscr=\Dscr(X/G)$. Then $\Dscr$ has a semi-orthogonal decomposition $\langle \dots, \D_{-2},\D_{-1},\D_0\ra$ with $\D_0\cong \D(\Lambda_n)$, $\D_{-i}\cong \D(\Lambda_j)$ for some $j<n$. 
\end{proposition}

\begin{proof}
Let us recall some relevant data from \cite[\S 5]{SVdB}. 
Let $L_i\in X(T)$ be given by $\diag(z_1,\dots,z_n)\to z_i$. 
The weights of $W=V^h\oplus (V^*)^h$ are $(\pm L_i)_i$, each weight occurring with multiplicity $h$, and a system of positive roots is given by $(L_i-L_j)_{i>j}$.

The weights of $W_\lambda$  are of the form $(\pm L_i)_{i\in S}$ for some subset $S\subset \{1,\dots,n\}$, each weight occurring with multiplicity $h$, and $Q_\lambda$ is isomorphic to $\GL(V_\lambda)$, where $V_\lambda$ is an $|S|$-dimensional vector space.
Thus,  $W_\lambda$ is a $Q_\lambda$-generic representation  (see \cite[\S 5.1]{SVdB}). 
Note that none of the weights of $W$ belongs to the subspace  of $X(T)_\RR$ spanned by the roots. Note also that $\varepsilon=\sum_i L_i\in X(T)_\RR$ is $\Wscr$-invariant and that the space of $\Wscr_\lambda$-invariant vectors is spanned by $\sum_{i\in S}L_i$.  
 Thus, it is enough to see that   $\sum_{i\in S}L_i$ satisfies \eqref{prazno} for $\lambda\in Y(T)^-$ in order to apply Corollary \ref{quasicor}.  Similarly as in the proof of \cite[Proposition 5.2.2]{SVdB} we have for $\Wscr_{G^\lambda}$-invariant $\nu\in X(T)_\RR$ 
\begin{multline*}
\nu-\bar{\rho_\lambda}+(1/2) (\bar{\Sigma}_{\lambda})_{\pm \epsilon}=\\
\sum_{i\in T_\lambda^+\cup T_\lambda^-}r_iL_i-\sum_{i\in T_\lambda^0} (n-2i+1)/2 L_i{}
+\left\{\sum_{i\in T_\lambda^0} a_iL_i\mid a_i\in -]h/2,h/2[\right\}=\\
\nu-\bar{\rho_\lambda}+(1/2) \bar{\Sigma}_{\lambda},
\end{multline*}
which implies  \eqref{prazno}.

Moreover, from the previous paragraph it also follows that the components of the semi-orthogonal decomposition from Proposition \ref{quasisod} are in this case isomorphic to the NCCRs of $Y_{j,h}$, $1\leq j\leq n$, or to $\D(k)$.
\end{proof}

\subsubsection{Pfaffian varieties}

Let $2n<h$, $W=V^h$, where $V$ is a $2n$-dimensional vector space equipped with a non-degenerate skew-symmetric bilinear form, $G={\rm Sp}_{2n}(k)$, and let $Y_{2n,h}^-=W\quot G$ be the variety of skew-symmetric $h\times h$-matrices of rank $\leq 2n$. 
We denote by $\Lambda_j$ the NCCR of $Y_{2j,h}$ given by \cite[Proposition 6.1.2]{SVdB}. For convenience we set $\Lambda_0:=k$.

\begin{proposition}
\label{prop:pfaffians}
If $h$ is odd then $\Dscr(X/G)$ has a semi-orthogonal decomposition $\langle \dots, \D_{-2},\D_{-1},\D_0\ra$ with $\D_0\cong \D(\Lambda_n)$, $\D_{-i}\cong \D(\Lambda_j)$ for some $j<n$. 
\end{proposition}

\begin{proof}
We recall some fragments of \cite[\S 6.1]{SVdB}. We assume that $(v_i)_i$ is a basis for $V$ such that the skew-symmetric form on $V$ is given by $\langle v_i,v_{i+n}\rangle=1$, $\langle v_i,v_j\rangle=0$ for $j\neq i\pm n$. Let $T\subset \Sp(V)$ be the maximal torus $\{\diag(z_1,\ldots,z_n,z_1^{-1},\ldots,z^{-1}_n)\}$ and let  $L_i\in X(T)$ be given by $(z_1,\dots,z_n)\to z_i$. 
The weights of $W=V^h$ are $(\pm L_i)_i$, each occurring with multiplicity $h$, and a system of positive roots is given by $(L_i-L_j)_{i>j}$, $(2 L_i)_i$,
$
\bar{\rho}=nL_1+(n-1)L_2+\cdots+L_n,
$
\[
\Sigma=\{\sum_i a_iL_i\mid a_i\in ]-h,h[\}.
\]
For  $\lambda\in Y(T)$ and $\Wscr_{G^\lambda}$-invariant $\nu\in X(T)_\RR$  we thus have
\[
\nu-\bar{\rho_\lambda}+(1/2) \bar{\Sigma}_{\lambda}=
\sum_{i\in T_\lambda^+\cup T_\lambda^-}r_iL_i-\sum_{i\in T_\lambda^0} (n-1+i)L_i
+\left\{\sum_{i\in T_\lambda^0} a_iL_i\mid a_i\in [-h/2,h/2]\right\}
\]
It easily follows that the boundary of this set does not intersect $X(T)$ if $h$ is odd, thus the condition \eqref{prazno} holds.

Similarly as in the case of determinantal varieties, $W_\lambda$ is $Q_\lambda$-generic, and here the semi-orthogonal decomposition consists of the NCCRs of Pfaffian varieties $Y_{2k,h}$, $1\leq k\leq n$, or $\D(k)$.
\end{proof}

\subsection{Proof of Proposition \ref{quasisod}}
\label{sec:quasimod}
\subsubsection{Preliminaries}
We remind the reader of our standing hypothesis that $W$ is quasi-symmetric.
In \S\ref{partition} (see \eqref{eq:partition}, Remark
\ref{rmk:faces}) we partitioned the set $X(T)^+$ according to the
relative interiors of faces of $-\bar\rho+r\bar\Sigma$, $r\geq
1$. However, in order to obtain a decomposition by NCCRs we need a
finer decomposition of $X(T)^+$. As indicated in Proposition
\ref{nccr} the decomposition parts should be roughly given by slightly
shifted faces of $-\bar\rho+(1/2)\bar{\Sigma}$. To obtain such a
decomposition we will inductively refine each face of
$-\bar\rho+r\bar\Sigma$. While in \S\ref{partition} we started with
$r\geq 1$, we need here $r> 1/2$. The following lemma will ensure that
this is indeed possible.

We note first that properties of the partition in \S\ref{partition}
remain basically unchanged if we replace $-\bar\rho+r\bar\Sigma$ by
$\nu-\bar\rho+r\bar\Sigma$ for a $\W$-invariant $\nu\in
X(T)_\RR$.

By replacing $-\bar{\rho}$ in the right hand side of
\eqref{expresschi} by $\nu-\bar\rho$ we obtain a minimal quadruple
which we denote by $(r_\chi^{\nu},{\bf S}_\chi^\nu)$ and also a
corresponding one-parameter subgroup $\lambda^\nu$ as in Lemma
\ref{ref-1.5}\eqref{tri}.  We will omit the extra decoration
$(-)^\nu$ in case no confusion can arise.

The following lemma is an improved and slightly generalized version of Corollary \ref{betaplus} 
which holds because  $W$ is quasi-symmetric.

\begin{lemma}\label{betaplus+}
Let $\chi\in X(T)^+$ be such that $r^{\nu}_\chi> 1/2$. 
If $p>0$ and $\mu=\chi+\beta_{i_1}+\cdots+\beta_{i_{p}}$, where $\{i_1,\ldots,i_{p}\}\subset\{1,\ldots,d\}$, $i_j\neq i_{j'}$ for
$j\neq j'$ and $\la\lambda^\delta,\beta_{i_j}\ra>0$,
then $({r}^\delta_{{\mu}^+},{\bf |S}^\nu_{{\mu}^+}|)<({ r}^\nu_\chi,{\bf |S}^{\nu}_\chi|)$.
Moreover, $\la\lambda^\nu,\chi\ra<\la\lambda^\nu,\mu^+\ra$.
\end{lemma}

\begin{proof}
For the first claim we proceed exactly as in  the proof of \cite[Theorem 1.6.1]{SVdB}. The second claim follows as the corresponding part in
Corollary \ref{betaplus}.
\end{proof}

Here we use notation introduced in \S\ref{sec:nonstable}.
For $\lambda\in Y(T)^-$ and for $\Wscr$-invariant $\nu\in X(T)_\RR$ and $\tau\in A$ we denote
\begin{align*}
\Lscr_{\lambda,r,\nu,\tau}&= X(\bar T)^\lambda_\tau\cap(\nu-\bar\rho_\lambda+r\Sigma_\lambda),\\
\Dscr_{\lambda,r,\nu}(X/G)_\tau&=\la P_{G^\lambda,\chi}\mid \chi\in \Lscr_{\lambda,r,\nu,\tau}\ra,\\
U_{\lambda,r,\nu,\tau}&=\bigoplus_{\mu\in \Lscr_{\lambda,r,\nu,\tau}}V_{G^\lambda}(\mu)
,\\
\Lambda_{\lambda,r,\nu,\tau}&=(\End(U_{\lambda,r,\nu,\tau})\otimes_k \Sym W_\lambda)^{G^\lambda}.
\end{align*}
so that in particular
\begin{equation}
\Dscr_{\lambda,r,\nu}(X/G)_\tau\cong \Dscr(\Lambda_{\lambda,r,\nu,\tau})
\end{equation}
We will also use some specializations of these notations in case 
part of the data $\lambda,r,\nu,\tau$ is omitted. If $\lambda$ is omitted
then we assume $\lambda=0$. If $r$ is omitted then we assume that $r=1/2+\epsilon$ where $\epsilon>0$ but  arbitrarily small.
For example with this convention we have
\begin{align*}
\Lscr_{\lambda,\nu,\tau}&= X(\bar T)^+_\tau\cap(\nu-\bar\rho_\lambda+(1/2)\bar{\Sigma}_\lambda),\\
\Lscr_{\nu,\tau}&= X(\bar T)^+_\tau\cap(\nu-\bar\rho+(1/2)\bar{\Sigma}).
\end{align*}
In the other direction, to indicate, context we may also write 
$(G,X,\lambda,r,\nu,\tau)$ instead of $(\lambda,r,\nu,\tau)$,
where we allow again $r$ or $\lambda$ to be omitted. 

\begin{proposition}\cite[Theorem 1.6.1]{SVdB}\label{sigmaq}
Assume $r>1/2$. Then  $\gldim \Lambda_{r,\nu,\tau}\allowbreak <\infty$.
Consequently, $\Dscr_{r,\nu}(X/G)_\tau\cong \D(\Lambda_{r,\nu,\tau})$.
In particular, (specializing to $\lambda=0$ and $r=1/2+\epsilon$) $\gldim \Lambda_{\nu,\tau}<\infty$ and $\Dscr_{\nu}(X/G)_\tau\cong \D(\Lambda_{\nu,\tau})$.
\end{proposition}

\begin{lemma}\label{tired} 
Let $\nu$ be a $\Wscr$-invariant element of $X(T)_\RR$. Using the partition \eqref{eq:partition} of $X(T)^+$, calculated with respect to $\nu-\bar{\rho}+r\bar{\Sigma}$, $r> 1/2$ and using $(r^\nu_\chi,\bold{S}^\nu_\chi)$, the analogues of Propositions \ref{affineso}, \ref{twist} hold. 
\end{lemma}

\begin{proof}
Using Lemma \ref{betaplus+} in place of Corollary \ref{betaplus} and Proposition \ref{sigmaq} it is easy to check that the proofs of the semi-orthogonal decompositions in \S\ref{sec:proofs} carry  over with the above partition of $X(T)^+$, 
 replacing 
$\Lambda_{<j}$, $\D_{<j}$, $\D_{j}$ by their shifted counterparts $\Lambda_{<j,\nu}$, $\D_{<j,\nu}$, $\D_{j,\nu}$. 
Consequently, the same holds true also for $\Lambda_{<j,\nu,\tau}$, $\D_{<j,\nu,\tau}$, $\D_{j,\nu,\tau}$ (see proof of Proposition \ref{twist}).
\end{proof}

If $K$ is a connected normal subgroup of $G$, and $Q$ is a chosen pseudo-comple\-ment, then we denote (as in the proof of Proposition \ref{twist}) by $\tilde T=T_K\times T_Q$ a maximal torus of $\tilde G=K\times Q$, such that $\tilde T\to T\subset G$ (under the multiplication map). Let us denote by $(\nu_K,\nu_Q)$ the image of $\nu\in X(T)$ in $X(\tilde T)$.

\begin{lemma}\label{faceTstable}
Let $\nu\in X(T)_\RR$ be $\Wscr$-invariant. 
Assume that there is a non-trivial connected subgroup $K$ of $G$ acting trivially on $X$. Let $Q$ be a pseudo-complement of $K$ in $G$. 
Then there exist a $\Wscr_Q$-invariant $\nu_Q\in X(T_Q)_\RR$,
a finite central group $A_Q$ of $Q$ acting trivially on $X$,
and $\tau_{Q}\in X(A_Q)$ such that
\[
\Lscr_{\lambda,r,\nu,\tau}\cong \Lscr_{Q_\lambda,X^\lambda,r,\nu_Q,\tau_Q}
\]
via the natural map $X(\bar{T})_\tau\to X(T)\to X(\tilde{T})\to X(T_Q)$.
 \end{lemma}

 \begin{proof}
 The proof is similar to
the proof of Proposition \ref{prop:nonstable1}, additionally noting  that $-\bar \rho$ decomposes as $(-\bar\rho_K,-\bar\rho_Q)$,   $\beta_i$ as $(0,\beta_i)$, $\nu$ as $(\nu_Q,\nu_K)$, and $\tau_Q\in X(A_Q)$  satisfies
$\tilde\tau_Q+\tilde\nu_K-\tilde{\bar{\rho}}_K=\tilde\tau$.
 \end{proof}

\subsubsection{Proof}
The proof is in two steps.
We first decompose $\D(X/G)$ such that its components are isomorphic to $\D(\Lambda_{\lambda,\nu,\tau})$, and then we decompose this further with components isomorphic to $\D(\Lambda_{\lambda,\nu,\tau}^\varepsilon)$.

\begin{lemma}\label{lemmaquasisod}
Let $X$ have a $T$-stable point.
There exist $\lambda_i\in Y(T)^-$,
finite central subgroups $A_i$ of $G^{\lambda_i}$ acting trivially on $X^{\lambda_i}$, $\tau_i\in X(A_i)$, $\W_{G^{\lambda_i}}$-invariant $\nu_i\in X(T)_\RR$ such that $\Dscr=\D(X/G)$ has a semi-orthogonal decomposition $\Dscr=\la\dots,\D_{-2},\D_{-1},\D_0\ra$ with $\D_{-i}\cong \D(\Lambda_{\lambda_i,\nu_i,\tau_i})$.
\end{lemma}

\begin{proof}
We decompose $\Dscr$ according to faces of $-\bar\rho+r\bar{\Sigma}$ 
assuming $r_j>1/2$. By Lemmas \ref{tired}, \ref{faceTstable} 
 and \eqref{eq:part2}
 we have 
 $\D_{-i}\cong \D_{r_{j_i},(\nu_{j_i})_{Q_{\lambda_{j_i}}}}(X^{\lambda_{j_i}}/Q_{\lambda_{j_i}})_{\tau_{Q_{\lambda_i}}}$.

As $\dim X^{\lambda_{j_i}}<\dim X$, $\dim Q_{\lambda_{j_i}}<\dim G$, $r_{j_i}>1/2$ we can decompose $\D_{-i}$ further. Note that in finitely many steps (as $\Lscr_{r,\nu,\tau}$ is finite) we reach the situation when every component of the decomposition is of the form $\la P_{\chi}\mid \chi\in\Lscr_{\lambda,\nu,\tau}\ra$ for some $\lambda\in Y(T)^-$, $\Wscr_{G^{\lambda}}$-invariant $\nu\in X(T)_\RR$, and $\tau\in X(A)$ for a finite central subgroup $A$ of $G^\lambda$ acting trivially on $X^\lambda$  by Lemma \ref{lem:technical} below (and the fact that for $r=1/2+\epsilon$ we have by definition $\Lscr_{\lambda,r,\nu,\tau}=\Lscr_{\lambda,\nu,\tau}$). \end{proof}

We now proceed to decompose
$\Dscr_{\lambda,\nu}(X/G)_\tau\cong \Dscr(\Lambda_{\lambda,\nu,\tau})$ further.

\begin{lemma}\label{0step}
There exist $\lambda_0=\lambda$, $\lambda_i\in Y(T)^-$, $\Wscr_{G^{\lambda_i}}$-invariant $\nu_i$, $\Wscr_{G^{\lambda_i}}$-invariant $\varepsilon_i\in X(T)_\RR$ which are weakly generic for $\Sigma_{\lambda_i}$, 
finite central subgroups $A_i$ of $G^{\lambda_i}$ acting trivially on $X^{\lambda_i}$,  $\tau_i\in X(A_i)$ such that $\D=\Dscr_{\lambda,\nu}(X/G)_\tau$ has a semi-orthogonal decomposition $\D=\la \D_{-k},\dots,\D_{-1},\D_0\ra$ with
 $\D_{-i}\cong \D(\Lambda_{\lambda_i,\nu_i,\tau_i}^{\varepsilon_i})$ for $i\geq 0$.
\end{lemma}

\begin{proof}
We choose a $\Wscr_{G^\lambda}$-invariant $\varepsilon\in X(T)_\RR$ weakly  generic for $\bar{\Sigma}_\lambda$ 
and small $a>0$ such that
\[
X(\bar T)_\tau^\lambda\cap(\nu-\bar\rho_\lambda+(1/2)(\bar{\Sigma}_\lambda)_\varepsilon)=
X(\bar T)_\tau^+\cap(\nu-\bar\rho_\lambda+a\varepsilon+(1/2)\bar{\Sigma}_\lambda)
\]
and
\begin{equation}\label{objem}
\Lscr_{\lambda,\nu,\tau}=X(\bar T)_\tau^\lambda\cap(\nu-\bar\rho_\lambda+a\varepsilon+r'{\Sigma}_\lambda)
\end{equation}
for some $r'>1/2$.
We take $\delta=\nu+a\varepsilon$, and let $\nu_Q$, $\tau_Q$ be as in Lemma \ref{faceTstable}, and partition $X(T_{Q_\lambda})^+$ accordingly.
Due to \eqref{objem} and Lemma \ref{tired}, $\Lscr_{\lambda,\nu,\tau}^\varepsilon\cong \cup_{i<j_1} F_i$ and there exists $k\geq 0$ such that $\Lscr^\epsilon_{0,\nu,\tau}\cup_{1\leq i\leq k} F_{j_i}= \Lscr_{\nu,\tau}$. We set $\varepsilon_0=\varepsilon$.

As in the second paragraph of the proof of Lemma \ref{lemmaquasisod}, we can decompose $\la P_\chi\mid \chi\in F_{j_i}\ra$ further, until the components are isomorphic to $\Dscr_{\lambda,\nu}(X/G)_\tau$  for some $\lambda\in Y(T)^-$, $\Wscr_{G^\lambda}$-invariant $\nu\in X(T)_\RR$ and $\tau\in X(A)$ for a finite central subgroup $A$ of $G^\lambda$ acting trivially on $X^\lambda$. Now we can repeat the first part of this proof and decompose  $\Dscr_{\lambda,\nu}(X/G)_\tau$ further. We get the desired decomposition by invoking Lemma \ref{lem:technical} below.
\end{proof}
We have used the following technical lemma.
\begin{lemma}\label{lem:technical}
Let $\lambda'\in Y(T)^-$, $\lambda''\in Y(T_{Q_{\lambda'}})^-$,
 let $\nu'\in X(T_{Q_{\lambda'}})$ be $\Wscr_{G^{\lambda''}}$-invariant, and let $A'$ be a finite central subgroup of $(Q_{\lambda'})^{\lambda''}$ acting trivially on $(X^{\lambda'})^{\lambda''}$, $\tau'\in X(A')$. There exist $\lambda\in Y(T)^-$, $\Wscr_{G^{\lambda}}$-invariant $\nu\in X(T)_\RR$, finite central subgroup $A$ of $G^\lambda$ acting trivially on $X^\lambda$, $\tau\in X(A)$,
  such that
\begin{align*}
\Lscr_{Q_{\lambda'},X^{\lambda'},\lambda'',r,\nu',\tau'}&\cong \Lscr_{\lambda,r,\nu,\tau},\\
\Dscr_{Q_{\lambda'},X^{\lambda'},\lambda'',r,\nu'}(X^{\lambda'}/Q_{\lambda'})_{\tau'}&\cong \D_{\lambda,r,\nu}(X/G)_\tau\cong \D_{r,\nu_Q}(X^\lambda/Q_\lambda)_{\tau_Q}.
\end{align*}
\end{lemma}

\begin{proof}
It is easy to check using Lemma \ref{faceTstable} that we can take $\lambda=a(\lambda'+b\lambda'')$ for small $b\in \QQ$, $a\in \NN$ such that $\lambda\in Y(T)^-$.
\end{proof}

\begin{proof}[Proof of Proposition \ref{quasisod}]
It suffices to combine Lemma \ref{lemmaquasisod} with Lemma \ref{0step}.
\end{proof}

\providecommand{\bysame}{\leavevmode\hbox to3em{\hrulefill}\thinspace}
\providecommand{\MR}{\relax\ifhmode\unskip\space\fi MR }
\providecommand{\MRhref}[2]{%
  \href{http://www.ams.org/mathscinet-getitem?mr=#1}{#2}
}
\providecommand{\href}[2]{#2}

\end{document}